%% file: frogsZd.tex
\documentclass[12pt,a4paper]{article} 

\usepackage[utf8]{inputenc}
\usepackage[T1]{fontenc}

\usepackage{amsmath}
\usepackage{amsthm}
\usepackage{amssymb}
\usepackage{amsfonts}
\usepackage{dsfont}
\usepackage{url}
\usepackage{graphicx}
\usepackage{hyperref}
\usepackage{color}
\usepackage{calc} 
\usepackage{tikz}
\usetikzlibrary{shapes.misc}
\usetikzlibrary{decorations.pathmorphing}
\usetikzlibrary{patterns}
\usepackage{pgfplots}
\pgfplotsset{compat=newest}
\usepackage{caption}
\usepackage{enumitem}
\setenumerate[1]{label=(\roman*)} 

\pgfdeclarepatternformonly{my crosshatch dots}{\pgfqpoint{-1pt}{-1pt}}{\pgfqpoint{5pt}{5pt}}{\pgfqpoint{6pt}{6pt}}%
{
    \pgfpathcircle{\pgfqpoint{0pt}{0pt}}{.5pt}
    \pgfpathcircle{\pgfqpoint{3pt}{3pt}}{.5pt}
    \pgfusepath{fill}
}

\makeatletter
\g@addto@macro\th@plain{\renewcommand*{\p@enumi}{\thethm~}}
\makeatother

\theoremstyle{plain}
\newtheorem{thm}{Theorem}[section]
\newtheorem{prop}[thm]{Proposition}
\newtheorem{lemma}[thm]{Lemma}
\newtheorem{con}[thm]{Conjecture}
\newtheorem{question}[thm]{Question}

\newtheorem{remark}[thm]{Remark}

\renewcommand{\P}{\mathbb{P}} 
\newcommand{\E}{\mathbb{E}} 
\newcommand{\N}{\mathbb{N}} 
\newcommand{\Z}{\mathbb{Z}} 
\newcommand{\R}{\mathbb{R}} 
\newcommand{\1}{\mathds{1}} 
\newcommand{\e}{\mathrm{e}} 

\newlength{\frogpathlength} 
\newlength{\frogpathheight}
\newcommand{\fp}[1]{\setlength{\frogpathheight}{\heightof{$\scriptstyle #1$}}
                    \,\begin{tikzpicture}[baseline=-0.6ex, node distance=5em,text height=\frogpathheight-.8ex,text depth=.25ex]
                        \setlength{\frogpathlength}{\maxof{\widthof{$\scriptstyle #1$}+.5em}{1em}}
                        \draw [->,decorate,decoration={zigzag,amplitude=.2ex,segment length=.35em,post=lineto,post length=0.2em}] (0,0) -- ++(\frogpathlength,0) 
                        node [above, midway, yshift=-0.2ex]{$\scriptstyle #1$};
                        \end{tikzpicture}\,
                    } 

\newcommand{\fm}{\operatorname{FM}} 
\newcommand{\fc}{F\hspace{-0.04em}C} 
\newcommand{\fb}{F\hspace{-0.1em}B} 

\newcommand{\Addresses}{{%
  \bigskip
  \footnotesize

   C.~Döbler, \textsc{Faculté des Sciences, de la Technologie et de la Communication, Université du Luxembourg, 
   Maison du Nombre, 6, Avenue de la Fonte, 4364 Esch-sur-Alzette, Luxembourg}\par\nopagebreak
  \textit{E-mail address:} \texttt{christian.doebler@uni.lu}
  
  \medskip
  
  N.~Gantert, T.~Höfelsauer, F.~Weidner \textsc{Fakultät für Mathematik, Technische Universität München, 
  Boltzmannstr. 3, 85748 Garching, Germany}\par\nopagebreak
  \textit{E-mail addresses:} \texttt{gantert@ma.tum.de}, \texttt{thomas.hoefelsauer@tum.de}, \newline \texttt{felizitas.weidner@tum.de}

  \medskip

  S.~Yu.~Popov, \textsc{Institute of Mathematics, Statistics and Scientific Computation, University of Campinas--UNICAMP, 
  rua Sérgio Buarque de Holanda 651, 13083--859, Campinas, SP, Brazil}\par\nopagebreak
  \textit{E-mail address:} \texttt{popov@ime.unicamp.br}

}}

\begin{document}
\title{Recurrence and Transience\\ of Frogs with Drift on $\Z^d$}
\author{Christian D\"obler, Nina Gantert, Thomas H\"ofelsauer, \\Serguei Popov and Felizitas Weidner}
\maketitle

\begin{abstract}
We study the frog model on $\Z^d$ with drift in dimension $d \geq 2$ and establish the existence of transient and recurrent regimes depending on the transition probabilities. We focus on a model in which the particles perform nearest neighbour random walks which are balanced in all but one direction. This gives a model with two parameters. We present conditions on the parameters for recurrence and transience, revealing interesting differences between dimension $d=2$ and dimension $d \geq 3$.
Our proofs make use of (refined) couplings with branching random walks for the transience, and with percolation for the recurrence.


\end{abstract}

\section{Introduction and main results}
\input{frogsZd_introduction}

\section{Preliminaries}\label{preliminaries}
\input{frogsZd_preliminaries}

\section{Proofs}\label{proofs}
In this section we prove the main results of the paper. To show recurrence we always compare the frog model with independent site percolation. To show transience we couple the frog model with branching random walks.
\input{frogsZd_recurrence_dgeq2_arbitrary_weight}

\input{frogsZd_recurrence_d2_arbitrary_drift}

\input{frogsZd_recurrence_dgeq3_arbitrary_drift}

\input{frogsZd_transience_dgeq2_arbitrary_drift}

\input{frogsZd_transience_d2_arbitrary_weight}

\section{Open Problems}\label{open_problems}
\input{frogsZd_outlook}

\paragraph{Acknowledgements} We thank Noam Berger for useful discussions and comments. We are grateful to the referee for pointing out several glitches in the first version.

\bibliographystyle{amsplain}
\bibliography{frogs}

\Addresses

\end{document}

%% file: frogsZd_introduction.tex
The frog model is a model of interacting random walks or, more generally, Markov chains on a graph $G=(V,E)$ in discrete time $\N_0$. It can be described as follows: There is one distinguished 
vertex $x_0\in V$, called the origin, and at time $0$ there is exactly one active particle (awake frog) at $x_0$. At every other vertex $x$, there is a (possibly zero) number $\eta_x$ of sleeping frogs. 
The frog at $x_0$ now starts walking randomly on the graph and each time it visits a site with sleeping frogs, they immediately become active and start performing random walks and waking up sleeping frogs themselves, independently of each other and of all other frogs. The transition mechanism of the individual frogs is the same for all frogs. The frog model is called recurrent if the probability that the origin $x_0$
is visited infinitely often equals $1$, otherwise the model is called transient. The frog model with $V=\Z^d$, $E$ the set of nearest-neighbour edges on $\Z^d$, $x_0:=0$, $\eta_x=1$ for each $x\in\Z^d\setminus\{0\}$ and the underlying random walk being simple random walk (SRW) on $\Z^d$ was studied by Telcs and Wormald \cite{TW99} who, however, called it egg model. The name frog model was only later suggested by Durrett. 
In \cite{TW99}, it is in particular shown that the frog model is recurrent for each dimension $d$. See also \cite{P01}.
Note that the frog model on $\Z^d$ with SRW is trivially recurrent for $d=1,2$, due to P\'{o}lya's theorem. Thus, in \cite{GS09} Gantert and Schmidt considered the frog model on $\Z$ with the underlying random walk having a drift to the right. They considered both fixed and i.i.d.~random initial configurations $(\eta_x)_{x\in\Z\setminus\{0\}}$ of sleeping frogs and derived a criterion separating transience from recurrence. In the case of an i.i.d.~initial configuration of sleeping frogs they also proved a zero-one law, which says that the probability of infinitely many returns to $0$ equals $1$ if $\E[\log^+(\eta)]=\infty$, and equals $0$ otherwise. Remarkably, this result only depends on the distribution of $\eta$ and does, in particular, not depend on the value of the drift. 
The recurrence part of the latter result was generalised to any dimension~$d$ by D\"obler and Pfeifroth in \cite{DP14}. They proved that the frog model on $\Z^d$ with underlying (irreducible) random walk which has an arbitrary drift to the right is recurrent provided that $\E[\log^+(\eta)^{\frac{d+1}{2}}]=\infty$. Another sufficient recurrence condition involving the tail behaviour of $\eta$ is derived in \cite{KZ17}. 
Kosygina and Zerner proved in \cite{KZ17} a zero-one law under the general condition that the frog trajectories are given by a transitive Markov chain. 
Recurrence and transience for the frog model on the $d$-ary tree have recently been investigated in \cite{HJJ14} and \cite{HJJ16} by Hoffman, Johnson and Junge. 
Other publications on the frog model include \cite{AMP02}, \cite{FMS04}, \cite{GNR17}, \cite{HW16}, \cite{JJ16} and \cite{JJ16sto} and \cite{R17} and references therein (the list is not exhaustive).

In the present article we study recurrence and transience of the frog model on $\Z^d$ for $d \geq 2$ when the underlying transition mechanism is not  symmetric. We assume that at each vertex in $\Z^d\setminus\{0\}$ there is exactly one sleeping frog at time $0$. 
Given this assumption, and using the zero-one law proved in \cite{KZ17}, one can now classify the transition laws of the particles in a recurrent and a transient class. Our proofs show that both regimes exist. In order to give more quantitative statements, we focus on a model in which the particles perform nearest neighbour random walks which are balanced in all but one direction. More precisely,
set $\mathcal{E}_d=\{\pm e_j \colon 1\leq j\leq d\}$ where $e_j$ denotes the $j$-th standard basis vector in $\R^d$, $j=1,\dotsc,d$. 
The particles move according to the following transition probabilities, which depend on two parameters $w \in [0,1]$ and $\alpha \in [0,1]$:
\begin{equation}\label{transition_function}
 \pi_{w,\alpha}(e) =
 \begin{cases}
  \frac{w(1+\alpha)}{2} & \text{for $e=e_1$} \\
  \frac{w(1-\alpha)}{2} & \text{for $e=-e_1$} \\
  \frac{1-w}{2(d-1)}    & \text{for $e \in\{\pm e_2,\dotsc, \pm e_d\}$} 
 \end{cases}
\end{equation}
The parameter $w$ is the weight of the drift axis $e_1$, i.e.~the random walk chooses to go in direction $\pm e_1$ with probability $w$. The parameter $\alpha$ describes the strength of the drift, i.e.~if the random walk has chosen to move in drift direction, it takes a step in direction $e_1$ with probability $\frac{1+\alpha}{2}$ and in direction $-e_1$ with probability $\frac{1-\alpha}{2}$. All other directions are balanced and equally probable.
Sometimes we need to consider the corresponding one-dimensional model where we have to demand $w=1$, i.e.~the transition probabilities are defined by $\pi_{\alpha}(e_1)=1-\pi_{\alpha}(-e_1)=\frac{1 + \alpha}{2}$. 
We denote the frog model on $\Z^d$ with parameters $w$ and $\alpha$ by $\fm(d,\pi_{w,\alpha})$.

First, let us discuss the extreme cases. For $w=1$ the frog model is one-dimensional and thus transient for any $\alpha \in (0,1]$ and recurrent for $\alpha=0$ by \cite{GS09}.
For $\alpha =1$ one easily checks that it is transient for any $w \in (0,1]$.
If $w=0$, then $\fm(d,\pi_{0, \alpha})$ is equivalent to the symmetric frog model in $d-1$ dimensions and hence recurrent.
If $\alpha =0$, we are back in the balanced case and the model is recurrent. This follows from Theorem~\ref{thm_d=2_arbitrary_weight_i} and Theorem~\ref{thm_d>2_arbitrary_weight} below.

In dimension $d=2$ the frog model is recurrent whenever $\alpha$ or $w$ are sufficiently small, i.e.~if the underlying transition mechanism is almost balanced. It is transient for $\alpha$ or $w$ close to $1$.

\begin{thm}\label{thm_d=2_arbitrary_weight}
 Let $d =2$ and $w \in (0,1)$. 
 \begin{enumerate}
  \item\label{thm_d=2_arbitrary_weight_i} There exists $\alpha_r = \alpha_r(w) > 0$ such that the frog model $\fm(d,\pi_{w,\alpha})$ is recurrent for all $0 \leq \alpha \leq \alpha_r$.
  \item\label{thm_d=2_arbitrary_weight_ii} There exists $\alpha_t = \alpha_t(w) < 1$ such that the frog model $\fm(d,\pi_{w,\alpha})$ is transient for all $\alpha_t \leq \alpha \leq 1$.
 \end{enumerate}
\end{thm}

\begin{thm}\label{thm_d=2_arbitrary_drift}
 Let $d=2$ and $\alpha \in (0,1)$.
 \begin{enumerate}
  \item\label{thm_d=2_arbitrary_drift_i} There exists $w_r = w_r(\alpha) > 0$ such that the frog model $\fm(d,\pi_{w,\alpha})$ is recurrent for all $0 \leq w \leq w_r$.
  \item\label{thm_d=2_arbitrary_drift_ii} There exists $w_t = w_t(\alpha) < 1$ such that the frog model $\fm(d,\pi_{w,\alpha})$ is transient for all $w_t \leq w \leq 1$.
 \end{enumerate}
\end{thm}

In dimension $d \geq 3$ the frog model is also recurrent if the transition probabilities are almost balanced. Further, for any fixed drift parameter $\alpha \in (0,1]$ it is transient if the weight $w$ is close to $1$. However, in contrast to $d=2$, for fixed $w \in [0,1)$ there is not always a transient regime. This follows from Theorem~\ref{thm_d>2_arbitrary_drift_i} below.

\begin{thm}\label{thm_d>2_arbitrary_weight}
 Let $d \geq 3$ and $w \in (0,1)$. 
 There exists $\alpha_r = \alpha_r(d,w) > 0$ such that the frog model $\fm(d,\pi_{w,\alpha})$ is recurrent for all $0 \leq \alpha \leq \alpha_r$.
\end{thm}

\begin{thm}\label{thm_d>2_arbitrary_drift}
 Let $d\geq 3$ and $\alpha \in (0,1)$.
 \begin{enumerate}
  \item\label{thm_d>2_arbitrary_drift_i} There exists $w_r > 0$, independent of $d$ and $\alpha$, such that the frog model $\fm(d,\pi_{w,\alpha})$ is recurrent for all $0 \leq w \leq w_r$.
  \item\label{thm_d>2_arbitrary_drift_ii} There exists $w_t = w_t(\alpha) < 1$, independent of $d$, such that the frog model $\fm(d,\pi_{w,\alpha})$ is transient for all $w_t \leq w \leq 1$.
 \end{enumerate}
\end{thm}

The results are graphically summarised in Figure~\ref{phase_diagram}. Note that the above theorems only make statements about the existence of recurrent, respectively transient regimes. We do not describe their shapes, as might be suggested by the curves depicted in Figure~\ref{phase_diagram}. For a discussion about their shape we refer the reader to Conjecture~\ref{con_critical_curve} at the end of this paper.

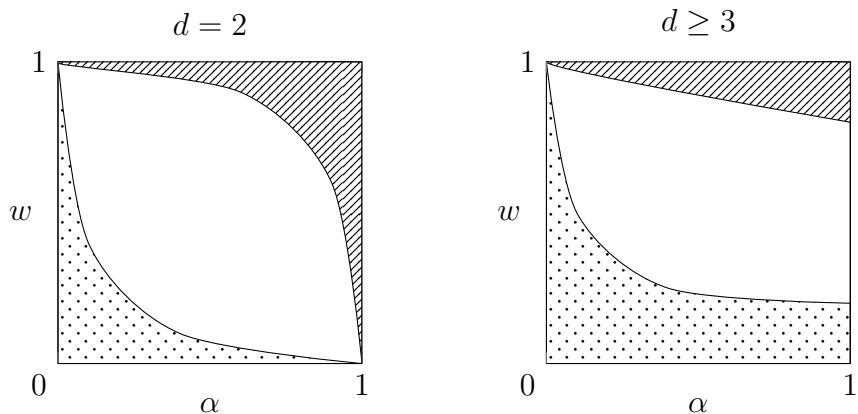
\begin{figure}[h]
\centering
\begin{tikzpicture}[baseline=0pt]
\begin{axis}[
    title = {$d = 2$},
    area legend,
    axis y line =box, 
    axis x line =box, 
    xtick={1},
    xticklabels={$1$},
    ytick={1},
    yticklabels={$1$},
    xlabel=$\alpha$,
    ylabel=$w$,
    every axis y label/.style={at={(ticklabel cs:0.5)},rotate=0}, 
    every axis x label/.style={at={(ticklabel cs:0.5)},rotate=0}, 
    extra x ticks={0},
    extra x tick labels={$0$},
    extra x tick style={xticklabel style={anchor=north east}},
    xmin=0,
    xmax=1,
    ymin=0,
    ymax=1,
    smooth,
    axis on top,
    x=4cm,
    y=4cm
    ]
\addplot[mark=none, pattern=my crosshatch dots, draw=black]  coordinates {(0,1) (0.1,0.4) (0.4, 0.1) (1,0) } \closedcycle; \label{recurrent} 
\addplot[mark=none, pattern=north east lines, draw=black]  coordinates {(1,0) (0.9, 0.6) (0.6,0.9) (0,1)  (1,1)} \closedcycle; \label{transient} 
\end{axis}                    
\end{tikzpicture}%
\hskip 40pt
\begin{tikzpicture}[baseline=0pt]
 \begin{axis}[
    title = {$d \geq 3$},
    axis y line =box, 
    axis x line =box, 
    xtick={1},
    xticklabels={$1$},
    ytick={1},
    yticklabels={$1$},
    xlabel=$\alpha$,
    ylabel=$w$,
    every axis y label/.style={at={(ticklabel cs:0.5)},rotate=0},
    every axis x label/.style={at={(ticklabel cs:0.5)},rotate=0},
    extra x ticks={0},
    extra x tick labels={$0$},
    extra x tick style={xticklabel style={anchor=north east}},
    xmin=0,
    xmin=0,
    xmax=1,
    ymin=0,
    ymax=1,
    smooth,
    axis on top,
    x=4cm,
    y=4cm
    ]
\addplot[mark=none, pattern=my crosshatch dots, draw=black]  coordinates {(0,1) (0.1,0.5) (0.4, 0.25) (1,0.2) } \closedcycle; 
\addplot[mark=none, pattern=north east lines, draw=black]  coordinates {(1,0.8) (0,1)  (1,1)} \closedcycle;
\end{axis}
\end{tikzpicture}%
\caption{Phase diagram for the frog model $\fm(d,\pi_{w,\alpha})$: the recurrent regime is marked by \ref{recurrent}, the transient one by \ref{transient}.}
\label{phase_diagram}
\end{figure}

These results show that, in contrast to $d=1$, recurrence and transience do depend on the drift in every dimension $d \geq 2$. This disproves the last conjecture in \cite{GS09} that some condition on the moments of $\eta$ would separate transience from recurrence as in the one-dimensional case.

The paper is organised as follows. In Section~\ref{preliminaries} we introduce notation used throughout the article, and collect some basic facts and results about random walks, percolation and the frog model, which are needed in the proofs. The proofs of the main results are presented in Section~\ref{proofs}. Further questions and conjectures are discussed in Section~\ref{open_problems}.

%% file: frogsZd_preliminaries.tex

\subsection*{Notation}
We refer to the frog model on $\Z^d$ with transition probabilities $\pi$ as $\fm(d, \pi)$. 
For $w, \alpha \in [0,1]$ and every vertex $x \in \Z^d$ let $(S_n^x)_{n \in \N_0}$ be a discrete time random walk on the lattice~$\Z^d$ starting at $x$ which moves according to the transition function $\pi_{w,\alpha}$ given by \eqref{transition_function}. Then $(S_n^x)_{n \in \N_0}$ describes the trajectory of the frog initially at vertex $x$. It starts to follow this trajectory once it is activated. We assume that the set $\{(S_n^x)_{n \in \N_0} \colon x \in \Z^d\}$ of random walks is independent, i.e.~active particles do not interact. Notice that this set of trajectories entirely determines the behaviour of the frog model. A formal definition of the frog model can be found in \cite{AMP02}. 
Note that $\pi_{1/d, 0}$ corresponds to a simple random walk on $\Z^d$. We write $\pi_{\text{sym}}$ in this case.

We refer to the frog that is initially at vertex $x\in \Z^d$ as ``frog~$x$''. 
We write $x \to y$ if frog $x$ (potentially) ever visits $y$, i.e.~$y \in \{S_n^x \colon n \in \N_0\}$.
For $x,y \in \Z^d$ and $A \subseteq \Z^d$ we say that there exists a frog path from $x$ to $y$ in $A$ and write $x \fp{A} y$ if there exist $n\in \N$ and $z_1, \ldots, z_n \in A$ such that $x \to z_1$, $z_i \to z_{i+1}$ for all $1\leq i < n$ and $z_n \to y$, or if $x \to y$ directly. Note that $x,y$ are not necessarily in $A$. Also the trajectories of the frogs $z_i$, $1 \leq i \leq n$, do not need to be in $A$.
For $x \in \Z^d$ we call the set 
\begin{equation}\label{def_frog_cluster}
\fc_x=\bigl\{y \in \Z^d \colon x \fp{\Z^d} y\bigr\} 
\end{equation}
the frog cluster of $x$. 
Note that, if frog $x$ ever becomes active, then every frog $y \in \fc_x$ is also activated. Observe that, as we only deal with recurrence and transience, the exact activation times are not important, but we are only interested in whether or not a frog is activated.

Further, we often use $(d-1)$-dimensional hyperplanes $H_n$ in $\Z^d$ defined by
\begin{equation}\label{definition_hyperplane}
H_n := \{x \in \Z^d \colon x_1=n\}
\end{equation}
for $n \in \Z$.


\subsection*{Some facts about random walks}

We need to deal with hitting probabilities of random walks on $\Z^d$. For $x,y \in \Z^d$ recall that $\{x\to y\}$ denotes the event that the random walk started at $x$ ever visits the vertex $y$. Analogously, for $A\subseteq \Z^d$ we write $\{x \to A\}$ for the event that the random walk started at $x$ ever visits a vertex in $A$.

\begin{lemma}\label{lemma_hitting_probability_SRW}
For $d \geq 3$ and $w \in (0,1)$ consider a random walk on $\Z^d$ with transition function~$\pi_{w,0}$. There exists a constant $c=c(d,w) >0$ such that for all $x\in \Z^d$
\begin{equation*}
\P(0\to x) \geq c \lVert x \rVert_2^{-(d-2)},
\end{equation*}
where $\lVert x \rVert_2 = \bigl(\sum_{i=1}^{d} x_i^2\bigr)^{1/2} $ is the Euclidean norm.
\end{lemma}

A proof of the lemma for the simple random walk, i.e.~with transition function $\pi_{\text{sym}}$, can e.g.~be found in \cite[Theorem~2.4]{AMP02} and \cite[Lemma~2.4]{AMP02pt}. The proof can immediately be generalised to our set-up using \cite[Theorem~2.1.3]{LL10}.

\begin{lemma}\label{lemma_hitting_probability_RW_drift}
For $d \geq 1$ and $\alpha,w \in (0,1)$ consider a random walk on $\Z^d$ with transition function~$\pi_{w,\alpha}$. Then for each $\gamma > 0$ there is a constant $c = c(d, \gamma, w, \alpha) >0$ such that for all $n \in \N$ and $x \in \Z^d$ with $x_1=-n$ and $\lvert x_i \rvert \leq \gamma\sqrt{n}$, $2 \leq i \leq d$, it holds that
\begin{equation*}
 \P(x \to 0) \geq c n^{-(d-1)/2}.
\end{equation*}
\end{lemma}

For a proof see e.g.~\cite[Lemma 3.1]{DP14}. 

\begin{lemma}\label{lemma_hitting_probability_hyperplane}
 For $d \geq 1$ and $\alpha,w \in (0,1]$ consider a random walk on $\Z^d$ with transition function~$\pi_{w,\alpha}$. Then for every $n \in \N$ and $H_{-n}$ as defined in \eqref{definition_hyperplane}
 \begin{equation*}
  \P(0\to H_{-n}) = \Bigl(\frac{1-\alpha}{1+\alpha}\Bigr)^n.
 \end{equation*}
\end{lemma}

\begin{proof}
 As $\P(0\to H_{-n}) = \P(0\to H_{-1})^n$ for $n \in \N$, it suffices to prove the lemma for $n =1$. By the Markov property
 \begin{equation*}
  \P(0\to H_{-1}) = \frac{1-\alpha}{2} + \frac{1+\alpha}{2} \P(0\to H_{-2}).
 \end{equation*}
 The result follows after a straightforward calculation.
\end{proof}


\subsection*{Some facts about percolation}

To prove recurrence we make use of the theory of independent site percolation on $\Z^d$ and therefore give a brief introduction here.
Let $p \in [0,1]$. Every site in $\Z^d$ is independently of the other sites declared open with probability $p$ and closed with probability $1-p$. An open cluster is a connected component of the subgraph induced by all open sites. It is well known that for $d \geq 2$ there is a critical parameter $p_c= p_c(d) \in (0,1)$ such that for all $p > p_c$ (supercritical phase) there is a unique infinite open cluster~$C$ almost surely, and for $p<p_c$ (subcritical phase) there is no infinite open cluster almost surely. Furthermore, denoting the open cluster containing the site~$x \in \Z^d$ by $C_x$, it holds that $\P(\lvert C_x \rvert=\infty)>0$ for $p>p_c$, and $\P(\lvert C_x \rvert=\infty)=0$ for $p<p_c$ and all $x \in \Z^d$. The following lemma states that the critical probability $p_c$ is small for $d$ large.

\begin{lemma}\label{lemma_pc_high_d}
For independent site percolation on $\Z^d$, 
\begin{equation*}
\lim_{d \to \infty} p_c(d) = 0.
\end{equation*}
\end{lemma}

Indeed, $p_c(d) = O\bigl(d^{-1}\bigr)$ holds. A proof of this result can e.g.~be found in \cite[Chapter~1, Theorem~7]{BR06}. Further, in the recurrence proofs we use the fact that an infinite open cluster is ``dense'' in $\Z^d$. The following weak version of denseness suffices.

\begin{lemma}\label{percolation_density}
Consider supercritical independent site percolation on $\Z^d$. There are constants $a,b>0$ such that 
\begin{equation*}
\P\bigl(\lvert A \cap C_x \rvert \geq a \lvert A \rvert \bigr) > b
\end{equation*}
for all $A \subseteq \Z^d$ and $x \in \Z^d$.
\end{lemma}

\begin{proof}
 
For $a>0$, $A \subseteq \Z^d$ and $x \in \Z^d$ the FKG-inequality yields
\begin{align*}
\P\bigl(\lvert A \cap C_x \rvert \geq a \lvert A \rvert \bigr)
& \geq \P\bigl( x \in C, \ \lvert A \cap C \rvert \geq a \lvert A \rvert \bigr)\\
& \geq \P( x \in C) \cdot \P\bigl(\lvert A \cap C \rvert \geq a \lvert A \rvert \bigr).
\end{align*}
Note that $\gamma:=\P( x \in C) \in (0,1)$ (and $\gamma$ does not depend on $x$) since the percolation is supercritical. By the Markov inequality
\begin{align*}
\P\bigl(\lvert A \cap C \rvert \geq a \lvert A \rvert \bigr)
& = 1 - \P\bigl(\lvert A \cap C^c \rvert \geq (1-a) \lvert A \rvert \bigr)\\
& \geq 1- \frac{\E \bigl[\lvert A \cap C^c \rvert \bigr] }{(1-a) \lvert A \rvert}\\
& = 1- \frac{1}{(1-a) \lvert A \rvert} \sum_{y \in A} \P(y \in C^c) \\
& = 1- \frac{1-\gamma}{1-a}>0,
\end{align*}
for $a$ small enough, which finishes the proof.
\end{proof}


\subsection*{Some results about frogs}

As mentioned in the introduction, the frog model presented in this paper satisfies a zero-one law, which is shown in \cite[Theorem~1]{KZ17} in a more general set-up. See also Appendix~A in \cite{KZ17} for a comment on the slightly different definition of recurrence used there.

\begin{thm}[\cite{KZ17}]\label{lemma_zero_one_law}
For any $d \geq 1$ and any nearest neighbour transition function $\pi$,
we have for $\fm(d,\pi)$ that the probability that the origin is visited infinitely many times by active frogs is either $0$ or $1$.
\end{thm}

Due to this zero-one law, to show recurrence, we only need to prove that the origin is visited infinitely often with positive probability. 

In the symmetric frog model the set of vertices visited by active frogs, rescaled by time, converges to a convex set. This shape theorem is proven by Alves et al.~in \cite[Theorem 1.1]{AMP02} and we use it in one of the proofs concerning recurrence.

\begin{thm}[\cite{AMP02}]\label{lemma_shape_theorem}
Consider $\fm(d,\pi_{\text{sym}})$ and let $\xi_n$ be the set of all sites visited by active frogs by time~$n$ and $\overline{\xi}_n := \{x + (-\frac12, \frac12]^d \colon x \in \xi_n\}$. Then there is a non-empty convex symmetric set $\mathcal{A}=\mathcal{A}(d) \subseteq \R^d$, $\mathcal{A} \neq \{0\}$, such that, for any $0 < \varepsilon < 1$
 \begin{equation*}
  (1- \varepsilon) \mathcal{A} \subseteq \frac{\overline{\xi}_n}{n} \subseteq (1+ \varepsilon) \mathcal{A}
 \end{equation*}
for all $n$ large enough almost surely. 
\end{thm}

\begin{remark} \label{remark_shape}
The proof of Theorem~\ref{lemma_shape_theorem} goes through for the ``lazy'' version of the frog model, where in each step a frog decides to stay where it is with probability $q \in (0,1)$, independently of all other frogs.
\end{remark}

Further, we need a result on the frog model with death. For $s \in [0,1]$ it is defined just as the usual frog model, but every active frog dies at every step with probability $1-s$ independently of everything else. The parameter $s$ is called the survival probability. We denote this frog model on $\Z^d$ by $\fm^*(d,\pi,s)$ if the underlying random walk has transition function $\pi$. Further, we denote frog clusters in the frog model with death by $\fc^*$, analogous to the notation introduced in \eqref{def_frog_cluster} for the frog model without death. In this paper we only use the frog model with death in the symmetric case, i.e. $\pi= \pi_{\text{sym}}$. We say that the frog model with death survives if at any time there is at least one active frog.
The frog model with death is intensively studied in \cite{AMP02pt} and also in \cite{FMS04} and \cite{LMP05}. We need the following lemma in the proofs concerning transience. 

\begin{lemma}\label{lemma_1d_fm}
 For $\fm(1,\pi_{1,\alpha})$ with $\alpha > 0$ and $\fm^*(1,\pi_{sym},s)$ with $s < 1$ there is $c>0$ such that $\P(0 \fp{\Z} -n) \leq \e^{-cn}$ for all $n \in \N$. 
\end{lemma}

\begin{proof}
 Let $p$ be the probability that a frog starting from $0$ ever hits the vertex $-1$. In both models we have $p <1$. Obviously, as $s <1$, this is true for $\fm^*(d,\pi_{\text{sym}},s)$. For $\fm(1,\pi_{1,\alpha})$ it follows from Lemma~\ref{lemma_hitting_probability_hyperplane}.

 For $n \in \N$ define $Y_n = \lvert\{m > -n \colon m \to -n \}\rvert$ if $-n \in \fc_0$, respectively $-n \in \fc_0^*$. Otherwise set $Y_n =0$. If $-n$ is visited by active frogs, then $Y_n$ counts the number of frogs to the right of $-n$ that potentially ever reach $-n$. The process $(Y_n)_{n \in \N}$ is a Markov chain on $\N_0$ with
 \begin{equation*}
  Y_{n+1} = 
  \begin{cases}
   0 & \text{if $Y_n = 0$,} \\
   \operatorname{Binomial}(Y_n+1,p) & \text{if $Y_n > 0$}.
  \end{cases}
 \end{equation*}
 Note that $\P(0 \fp{\Z} -n) = \P(Y_n >0)$ by definition.
 A straightforward calculation shows that there is $k_0 \in \N$ such that $\P(Y_{n+1} < Y_n \mid Y_n = k) > \frac23$ for all $k \geq k_0$. Hence, we can dominate the Markov chain $(Y_n)_{n \in \N}$ by the Markov chain $(\widetilde{Y}_n)_{n \in \N}$ on $\{0, k_0, k_0+1, \ldots\}$ with transition probabilities 
   \begin{align*}
   \P(\widetilde{Y}_{n+1} = l \mid \widetilde{Y}_n =k) = 
   \begin{cases}
    \frac{1}{3} & \text{if $l=k+1$, $k > k_0$}, \\
    \frac{2}{3} & \text{if $l=k-1$, $k > k_0$}, \\
    (1-p)^{k_0+1} & \text{if $l=0$, $k = k_0$}, \\
    1-(1-p)^{k_0+1} & \text{if $l=k+1$, $k = k_0$}, \\
    1 & \text{if $l=k=0$}
   \end{cases}
  \end{align*}
 for all $n \in \N$ and starting point $\widetilde{Y}_1 = \max\{Y_1, k_0\}$. Obviously, we have $\P(Y_n >0) \leq \P(\widetilde{Y}_n >0)$ for all $n \in \N$.
 Let $T_k= \min\{n \in \N \colon \widetilde{Y}_n = k\}$ and $T_{k,l}=T_l - T_k$. Note that $\P(\widetilde{Y}_n >0) = \P(T_0 >n)$. For $t >0$,
we apply the Markov inequality and use the strong Markov property to get
 \begin{align}\label{proof_lemma_1d_fm_1}
 \P(T_0 > n) &=    \P\biggl(\sum_{k=k_0}^{\widetilde{Y}_1-1} T_{k+1,k} + T_{k_0,0} > n\biggr) \nonumber\\
             &\leq \e^{-tn}\E\biggl[\exp\biggl(t \sum_{k=k_0}^{\widetilde{Y}_1-1} T_{k+1,k} + tT_{k_0,0}\biggr)\biggr] \nonumber\\
             &=    \e^{-tn} \sum_{l=k_0}^{\infty} \prod_{k=k_0}^{l-1} \E\bigl[\exp(tT_{k+1,k})\bigr] \E\bigl[\exp( tT_{k_0,0})\bigr]\P(\widetilde{Y}_1 = l) \nonumber\\
             &=    \e^{-tn} \sum_{l=0}^{\infty} \E\bigl[\exp(tT_{k_0+1,k_0})\bigr]^l \E\bigl[\exp( tT_{k_0,0})\bigr]\P(\widetilde{Y}_1 = l+k_0).
 \end{align}
$\widetilde{Y}_1$ can only be equal to $l+k_0$ if at least one frog to the right of $l-1$ reaches $-1$. Thus, 
\begin{equation}\label{proof_lemma_1d_fm_2}
 \P(\widetilde{Y}_1 = l+k_0) \leq \sum_{i=l}^{\infty} p^{i+1} = p^l \frac{p}{1-p}. 
\end{equation}
Now, we choose $t>0$ small enough such that $\E\bigl[\exp(tT_{k_0+1,k_0})\bigr] < p^{-1}$. Then \eqref{proof_lemma_1d_fm_2} shows that the sum in \eqref{proof_lemma_1d_fm_1} is finite, which yields the claim.
\end{proof}


\subsection*{A lemma on Bernoulli random variables}

We will repeatedly use the following simple lemma. Note that the random variables in this lemma do not have to be independent.

\begin{lemma}\label{lemma_sum_rv}
For $i \in \N$ let $X_i$ be a Bernoulli($p_i$) random variable with $\inf_{i\in \N}p_i =:p >0$. Then for every $a >0$ and $n \in \N$ 
\begin{equation*}
\P \left(\frac{1}{n} \sum_{i=1}^n X_i \geq a \right) \geq p-a. 
\end{equation*}
\end{lemma}

\begin{proof}
Since $\E[X_i]\geq p$ and $\frac{1}{n} \sum_{i=1}^n X_i \leq 1$, we have
\begin{align*}
p \leq \E\left[\frac{1}{n}\sum_{i=1}^n X_i \right] \leq \P \left(\frac{1}{n}\sum_{i=1}^n X_i \geq a \right) + a ,
\end{align*}
which yields the claim.
\end{proof}

%% file: frogsZd_recurrence_dgeq2_arbitrary_weight.tex

\subsection*{Recurrence for $d \geq 2$ and arbitrary weight}
In this section we prove Theorem \ref{thm_d=2_arbitrary_weight_i} and Theorem \ref{thm_d>2_arbitrary_weight}. Throughout this section assume that $w <1$ is fixed. 
To illustrate the basic idea of the proof we first sketch it for $d=2$. We call a site $x$ in $\Z^2$ open if the trajectory $(S_n^x)_{n \in \N_0}$ of frog $x$ includes the four neighbouring vertices $x \pm e_1, x \pm e_2$ of $x$, i.e.~if $x \to x \pm e_1$ and $x \to x \pm e_2$. Note that for this definition it does not matter whether frog $x$ is ever activated or not. All sites are open independently of each other due to the independence of the trajectories of the frogs. Furthermore, the probability of a site to be open is the same for all sites. Consider the percolation cluster $C_0$ that consists of all sites that can be reached from $0$ by open paths, i.e.~paths containing only open sites. Note that all frogs in $C_0$ are activated as frog $0$ is active in the beginning. In this sense the frog model dominates the percolation.
As we are in $d=2$, the probability of a site $x$ being open equals $1$ for $\alpha=0$ and by continuity is close to $1$ if $\alpha$ is close to $0$. Thus, if $\alpha$ is close enough to $0$ the percolation is supercritical. 
Hence, with positive probability the cluster $C_0$ containing the origin is infinite. By Lemma~\ref{percolation_density} this infinite cluster contains many sites close to the negative $e_1$-axis. This shows that many frogs that are initially close to this axis get activated. Each of them travels in the direction of the $e_1$-axis and has a decent chance of visiting $0$ on its way. Hence, this will happen infinitely many times. This argument shows that the origin is visited by infinitely many frogs with positive probability. Using the zero-one law stated in Theorem~\ref{lemma_zero_one_law} yields the claim.

In higher dimensions the probability of a frog to visit all its neighbours is not close to $1$ however small the drift may be. We can still make the reasoning work by using a renormalization type argument. 
To make this argument precise let $K$ be a non-negative integer that will be chosen later. We tessellate $\Z^d$ for $d \geq 2$ with cubes $(Q_x)_{x \in \Z^d}$ of size $(2K+1)^d$. For every $x \in \Z^d$ we define
\begin{align}\label{def_box}
\begin{split}
 q_x &= q_x(K) = (2K+1)x,\\
 Q_x &= Q_x(K) = \{y \in \Z^d \colon \lVert y-q_x \rVert_{\infty} \leq K\}, 
\end{split}
\end{align}
where $\lVert x \rVert_{\infty} = \max_{1 \leq i \leq d}{\lvert x_i \rvert}$ is the supremum norm. 
We call a site $x \in \Z^d$ open if for every $e \in \mathcal{E}_d$ there exists a frog path from $q_x$ to $q_{x + e}$ in $Q_x$. Otherwise, $x$ is said to be closed. 
The probability of a site $x$ to be open does not depend on $x$, but only on the drift parameter $\alpha$ and the cube size $K$. We denote it by $p(K, \alpha)$. This defines an independent site percolation on $\Z^d$, which, as mentioned before, is dominated by the frog model in the following sense: For any $x \in C_0$ the frog at $q_x$ will be activated in the frog model, i.e.~$q_x \in \fc_0$ with $\fc_0$ as defined in \eqref{def_frog_cluster}. 

In the next two lemmas we show that the probability $p(K,\alpha)$ of a site to be open is close to $1$ if the drift parameter $\alpha$ is small and the cube size $K$ is large. We first show this claim for the symmetric case $\alpha=0$.

\begin{lemma} \label{lemma_recurrence_cube_size}
For every $w<1$ in the frog model $\fm(d, \pi_{w,0})$ we have
 \begin{equation*}
  \lim_{K \to \infty} p(K, 0) =1.
 \end{equation*}
\end{lemma}

\begin{proof}
For $d=2$ we obviously have $p(K, 0) = 1$ for all $K \in \N_0$ as balanced nearest random walk on $\Z^2$ is recurrent. Therefore, we can assume $d \geq 3$. The proof of the lemma relies on the shape theorem (Theorem~\ref{lemma_shape_theorem}) for the frog model. This theorem assumes equal weights on all directions. As in our model the $e_1$-direction has a different weight, we need a workaround. We couple our model with a modified frog model on $\Z^{d-1}$ in which the frogs in every step stay where they are with probability $w$ and move according to a simple random walk otherwise. A direct coupling shows that, up to any fixed time, in the modified frog model on $\Z^{d-1}$ there are at most as many frogs activated as in the frog model $\fm(d,\pi_{w,0})$. Note that Theorem~\ref{lemma_shape_theorem} holds true for the modified frog model on $\Z^{d-1}$, see Remark~\ref{remark_shape}. Let $\xi_K$, respectively $\xi_K^{\text{mod}}$, be the set of all sites visited by active frogs by time~$K$ in the frog model $\fm(d,\pi_{w,0})$, respectively the modified frog model on $\Z^{d-1}$. Further, let $\overline{\xi_K^{\text{mod}}} := \{x + (-\frac12, \frac12]^{d-1} \colon x \in \xi_K^{\text{mod}}\}$. By Theorem~\ref{lemma_shape_theorem} there exists a non-trivial convex symmetric set $\mathcal{A}=\mathcal{A}(d) \subseteq \R^{d-1}$ and an almost surely finite random variable~$\mathcal{K}$ such that 
 \begin{equation*}
   \mathcal A  \subseteq \frac{\overline{\xi_K^{\text{mod}}}}{K}
 \end{equation*}
for all $K \geq \mathcal{K}$. This implies that there is a constant $c_1 = c_1(d) > 0$ such that $\lvert \xi_K^{\text{mod}}\rvert \geq c_1 K^{d-1}$ for all $K \geq \mathcal{K}$. By the coupling the same statement holds true for $\xi_K$.
As $\xi_K \subseteq Q_0(K)$ and any vertex in $\xi_K$ can be reached by a frog path from $0$ in $Q_0$, this implies
\begin{equation*}
 \Bigl\lvert\Bigl\{y \in Q_0\colon 0 \fp{Q_0} y\Bigr\} \Bigr\rvert \geq \lvert \xi_K\rvert \geq c_1 K^{d-1}
\end{equation*}
for all $K \geq \mathcal{K}$. Thus we have at least $c_1 K^{d-1}$ vertices in the box $Q_0$ that can be reached by frog paths from $0$. Each frog in $Q_0$ has a chance to reach the centre $q_{e}$ of a neighbouring box. More precisely, by Lemma~\ref{lemma_hitting_probability_SRW} there is a constant $c_2 =c_2(d)>0$ such that
\begin{equation} \label{proof_lemma_recurrence_cube_size_0}
 \P \bigl( y \to q_{e} \bigr) \geq \frac{c_2}{K^{d-2}}
\end{equation} 
for any vertex $y \in Q_0$ and $e \in \mathcal{E}_d$.
Hence, for any $e \in \mathcal{E}_d$
\begin{align} \label{proof_lemma_recurrence_cube_size_1}
 \P\bigl( (0 \fp{Q_0} q_{e})^c  \mid K \geq \mathcal{K} \bigr) 
	&= \P \Bigl( \bigl\{y \not\to q_{e} \text{ for all } y \in Q_0 \text{ with } 0 \fp{Q_0} y\bigr\}  \bigm\vert K \geq \mathcal{K} \Bigr) \nonumber\\
        &\leq \Bigl(1-\frac{c_2}{K^{d-2}}\Bigr)^{c_1K^{d-1}} \nonumber\\
        & \leq \e^{-c_1c_2K},
\end{align}
where we used for the first inequality the fact that a frog moves independently of all frogs in $Q_0$ once it will never return to $Q_0$ and the uniformity of the bound in \eqref{proof_lemma_recurrence_cube_size_0}. Therefore,
\begin{align} \label{proof_lemma_recurrence_cube_size_2}
 p(K,0) &\geq  \P\Bigl(\bigcap_{e \in \mathcal{E}_d} \{0 \fp{Q_0} q_{e} \} \Bigm\vert K \geq \mathcal{K} \Bigl) \P_{0}(K \geq \mathcal{K}) \nonumber\\
        &\geq  \biggl[ 1- 2d \operatorname{e}^{-c_1 c_2 K} \biggr] \P(K \geq \mathcal{K}).
\end{align}
Since $\mathcal K$ is almost surely finite, we have $\lim_{K \to \infty}\P_{0}(K \geq \mathcal{K}) =1$. Thus, the right hand side of \eqref{proof_lemma_recurrence_cube_size_2} tends to $1$ in the limit $K\to \infty$.
\end{proof}

\begin{lemma}\label{lemma_recurrence_small_drift}
For fixed $w <1$, in the frog model $\fm(d,\pi_{w,\alpha})$ we have for all $K \in \N_0$
 \begin{equation*}
  \liminf_{\alpha \to 0} p(K, \alpha) \geq p(K,0).
 \end{equation*}
\end{lemma}

\begin{proof}
Let $L(a,b,c,K)$ be the number of possible realizations such that all $q_{x \pm e}$, $e \in \mathcal{E}_d$, are visited by frogs in $Q_0$ for the first time after in total (of all frogs) exactly $a$ steps in $e_1$-direction, $b$ steps in $-e_1$-direction and $c$ steps in all other directions. Note that $L(a,b,c,K)$ is independent of $\alpha$. We have
\begin{equation*}
p(K, \alpha) =   \sum_{a,b,c=1}^\infty L(a,b,c,K) \biggl(\frac{w(1+\alpha)}{2}\biggr)^a \biggl(\frac{w(1-\alpha)}{2}\biggr)^b \biggl(\frac{1-w}{2(d-1)}\biggr)^c.
\end{equation*} 
The claim now follows from Fatou's Lemma.
\end{proof}

\begin{proof}[Proof of Theorem \ref{thm_d=2_arbitrary_weight_i} and Theorem \ref{thm_d>2_arbitrary_weight}]
By Lemma \ref{lemma_recurrence_cube_size} and Lemma \ref{lemma_recurrence_small_drift} we can assume that $K$ is big enough and $\alpha >0$ small enough such that $p(K, \alpha)> p_c(d)$, i.e.~the percolation with parameter $p(K, \alpha)$ on $\Z^d$ constructed at the beginning of this section is supercritical. 

Consider boxes $B_n = \{-n\} \times [-\sqrt{n},\sqrt{n}]^{d-1}$ for $n \in \N$. By Lemma~\ref{percolation_density} there are constants $a,b > 0$ and $N \in \N$ such that for all $n \geq N$
\begin{equation*}
 \P(\lvert B_n \cap C_0\rvert \geq a n^{(d-1)/2})>b.
\end{equation*}
After rescaling, the boxes $B_n$ correspond to the boxes
\begin{equation*}
\fb_n = \{y \in \Z^d \colon \lvert y_1 + (2K+1)n \rvert \leq K,\, \lvert y_i \rvert \leq (2K+1)\sqrt{n} +K,\, 2 \leq i \leq d\}.
\end{equation*}
Recall that $\fc_0$ consists of all vertices reachable by frog paths from $0$ as defined in \eqref{def_frog_cluster}, and note that $x \in B_n \cap C_0$ implies $q_x \in  \fb_n \cap \fc_0$. 
This shows
\begin{equation}\label{proof_thm_recreg_1}
 \P(\lvert \fb_n \cap \fc_0 \lvert \geq a n^{(d-1)/2})>b
\end{equation}
for $n$ large enough. Analogously to \eqref{proof_lemma_recurrence_cube_size_1}, by Lemma~\ref{lemma_hitting_probability_RW_drift} and \eqref{proof_thm_recreg_1} the probability that at least one frog in $\fb_n$ is activated and reaches $0$ is at least
\begin{equation*}
 \Bigl(1-(1-cn^{-(d-1)/2})^{an^{(d-1)/2}}\Bigr)b \geq \bigl(1 - \e^{-ac}\bigr)b,
\end{equation*}
where $c=c(K,d,w)>0$ is a constant. Altogether we get by Lemma~\ref{lemma_sum_rv}
\begin{align*}
 \P(\text{$0$ visited infinitely often}) &=    \lim_{n \to \infty} \P(\text{$0$ is visited $\varepsilon n$ many times }) \\
                                         &\geq \liminf_{n \to \infty} \P\biggl( \sum_{i=1}^n \1_{\{\exists x \in \fb_i \cap \fc_{0} \colon x \to 0 \}} \geq \varepsilon n \biggr) \\
                                         &\geq \bigl(1 - \e^{-ac}\bigr)b - \varepsilon > 0
\end{align*}
for $\varepsilon$ sufficiently small. The claim now follows from Theorem~\ref{lemma_zero_one_law}. 
\end{proof}

%% file: frogsZd_recurrence_d2_arbitrary_drift.tex
\subsection*{Recurrence for $d = 2$ and arbitrary drift}

In this section we prove Theorem~\ref{thm_d=2_arbitrary_drift_i}. Throughout the section let $\alpha < 1$ be fixed.
We couple the frog model with independent site percolation on $\Z^2$. Let $K$ be an integer that will be chosen later. We tessellate $\Z^2$ with segments $(Q_x)_{x \in \Z^2}$ of size $2K+1$. For every $x = (x_1, x_2) \in \Z^2$ we define
\begin{align*}
 q_x &= q_x(K) = \bigl( x_1,  (2K+1)x_2\bigr), \\
 Q_x &= Q_x(K) = \{y \in \Z^2 \colon y_1 = x_1, \lvert y_2-(2K+1)x_2 \rvert \leq K\}.
\end{align*}
We call the site $x \in \Z^2$ open if there are frog paths from $q_x$ to $q_{x+e}$ in $Q_x$ for all $e \in \mathcal{E}_2$. As before, we denote the probability of a site to be open by $p(K,w)$. Note that this probability does not depend on $x$.

\begin{lemma}\label{lemma_d_2_arbitrary_drift_percolation_parameter_bound}
 For $\alpha <1$, in the frog model $\fm(2,\pi_{w,\alpha})$ we have
 \begin{equation*}
  \lim_{K \to \infty} \liminf_{w \to 0} p(K,w) =1.
 \end{equation*}
\end{lemma}

\begin{proof}
We claim that there is a constant $c=c(\alpha)>0$ such that for any $K \in \N_0$ and $x \in Q_0$
\begin{equation}\label{proof_lemma_d_2_arbitrary_drift_percolation_parameter_bound_1}
 \liminf_{w \to 0} \P\Bigr(\bigcap_{ e \in \mathcal{E}_2} \{x \to q_{e}\} \Bigl) \geq c.
\end{equation}
We can estimate the probability in \eqref{proof_lemma_d_2_arbitrary_drift_percolation_parameter_bound_1} by
\begin{equation*}
\P\Bigr(\bigcap_{e \in \mathcal{E}_2} \{x \to q_{e}\} \Bigl) \geq \P\bigl(x \to q_{-e_2} \bigr) \P\bigl(q_{-e_{2}} \to q_{-e_1} \bigr) \P\bigl(q_{-e_{1}} \to q_{e_2} \bigr) \P\bigl(q_{e_{2}} \to q_{e_1} \bigr).
\end{equation*}
The probability of moving in $\pm e_2$-direction for $\lceil w^{-1} \rceil$ steps is $(1-w)^{\lceil w^{-1} \rceil}$. Conditioning on moving in this way, we just deal with a simple random walk on $\Z$. There exists a constant $c_1>0$ such that this random walk hits $-K$ within $\lceil w^{-1} \rceil$ steps with probability at least $c_1$ for all $w$ close to $0$.
Therefore,
\begin{equation} \label{proof_lemma_d_2_arbitrary_drift_percolation_parameter_bound_2}
\P\bigl(x \to q_{-e_2} \bigr) \geq c_1 (1-w)^{\lceil w^{-1} \rceil} 
                              \geq \frac{c_1}{4}.
\end{equation}
The probability of moving exactly once in $-e_1$-direction and otherwise in $\pm e_2$-direction within $\lceil w^{-1} \rceil+1$ steps is
\begin{equation*}
\bigl(\lceil w^{-1} \rceil +1\bigr) \frac{(1-\alpha)w}{2} (1-w)^{\lceil w^{-1} \rceil} \geq \frac{1-\alpha}{8}
\end{equation*}
for $w$ close to $0$. Therefore, analogously to \eqref{proof_lemma_d_2_arbitrary_drift_percolation_parameter_bound_2} there exists a constant $c_2>0$ such that
\begin{equation*}
\P\bigl(q_{-e_2} \to q_{-e_1} \bigr) \geq \frac{c_2(1-\alpha)}{8}
\end{equation*}
for $w$ sufficiently close to $0$. The two remaining probabilities $\P\bigl(q_{-e_{1}} \to q_{e_2} \bigr)$ and $\P\bigl(q_{e_{2}} \to q_{e_1} \bigr)$ can be estimated analogously, which implies \eqref{proof_lemma_d_2_arbitrary_drift_percolation_parameter_bound_1}.

If frog $0$ activates all frogs in $Q_0$ and any of these $2K$ frogs manages to visit the centres of all neighbouring segments, then $0$ is open. By independence of the trajectories of the individual particles in $Q_0$ this implies
\begin{equation}\label{proof_lemma_d_2_arbitrary_drift_percolation_parameter_bound_3}
 p(K,w) \geq \P\Bigl( \bigcap_{x \in Q_0} \{0 \to x\} \Bigr) \biggl(1- \Bigl(1-\P\Bigl(\bigcap_{1\leq i \leq 4} \{x \to q_{e_i}\}\Bigr)\Bigr)^{2K} \biggr).
\end{equation}
As in the proof of Lemma~\ref{lemma_recurrence_small_drift} one can show that for $w \to 0$ the first factor in \eqref{proof_lemma_d_2_arbitrary_drift_percolation_parameter_bound_3} converges to $1$. Therefore, taking limits in \eqref{proof_lemma_d_2_arbitrary_drift_percolation_parameter_bound_3} and using \eqref{proof_lemma_d_2_arbitrary_drift_percolation_parameter_bound_1} yields the claim.
\end{proof}

\begin{proof}[Proof of Theorem~\ref{thm_d=2_arbitrary_drift_i}]
By Lemma~\ref{lemma_d_2_arbitrary_drift_percolation_parameter_bound} we can choose $K$ big and $w$ small enough such that $p(K,w) > p_c(2)$, where $p_c(2)$ is the critical parameter for independent site percolation on $\Z^2$. As in the proof of Theorem~\ref{thm_d=2_arbitrary_weight_i} and Theorem~\ref{thm_d>2_arbitrary_weight} the coupling with supercritical percolation now yields recurrence of the frog model. As we rescaled the lattice $\Z^2$ slightly different this time, the box $B_n$ defined in the proof of Theorem~\ref{thm_d=2_arbitrary_weight_i} and Theorem~\ref{thm_d>2_arbitrary_weight} now corresponds to the box
\begin{equation*}
\fb_n = \{y \in \Z^2 \colon y_1 =-n,\, \lvert y_2\rvert \leq (2K+1)\sqrt{n} +K\}.
\end{equation*}
Since only asymptotics in $n$ matter for the proof, it otherwise works unchanged.
\end{proof}

%% file: frogsZd_recurrence_dgeq3_arbitrary_drift.tex

\subsection*{Recurrence for arbitrary drift and $d \geq 3$}

The proof of Theorem \ref{thm_d>2_arbitrary_drift_i} again relies on the idea of comparing the frog model with percolation. But instead of looking at the whole space $\Z^d$ as in the previous proofs, we consider a sequence of $(d-1)$-dimensional hyperplanes $(H_{-n})_{n \in \N_0}$ with $H_{-n}$ as defined in \eqref{definition_hyperplane}. We compare the frogs in each hyperplane with supercritical percolation, ignoring the frogs once they have left their hyperplane and all the frogs from other hyperplanes. Within a hyperplane we now deal with a frog model without drift, but allow the frogs to die in each step with probability $w$ by leaving their hyperplane, i.e.~we are interested in $\fm^*(d-1,\pi_{\text{sym}},1-w)$. Hence, the argument does not depend on the value of the drift parameter $\alpha<1$. 

We start with one active particle in the hyperplane $H_0$. With positive probability this particle initiates an infinite frog cluster in $H_0 $ if $w$ and therefore the probability to leave the hyperplane is sufficiently small. Every frog eventually leaves $H_0$ and has for every $n \in \N$ a positive chance of activating a frog in the hyperplane $H_{-n}$, which might start an infinite cluster there. This is the only time where we need $\alpha <1$ in the proof of Theorem~\ref{thm_d>2_arbitrary_drift_i}. Using the denseness of such clusters we can then proceed as before.

We split the proof of Theorem~\ref{thm_d>2_arbitrary_drift_i} into two parts:

\begin{prop}\label{prop_d>2_arbitrary_drift_large_d}
 There is $d_0 \in \N$ and $w_r > 0$, independent of $d$ and $\alpha$, such that the frog model $\fm(d,\pi_{w,\alpha})$ is recurrent for all $0 \leq w \leq w_r$, $0 \leq \alpha < 1$ and $d \geq d_0$.
\end{prop}

\begin{prop}\label{prop_d>2_arbitrary_drift_small_d}
 For every $d\geq 3$ there is $w_r = w_r(d) > 0$, independent of $\alpha$, such that the frog model $\fm(d,\pi_{w,\alpha})$ is recurrent for all $0 \leq w \leq w_r$ and all $0 \leq \alpha < 1$.
\end{prop}

We first prove Proposition~\ref{prop_d>2_arbitrary_drift_large_d}.
As indicated above we need to study the frog model with death and no drift in $\Z^{d-1}$. To increase the readability of the paper let us first work in dimension $d$ instead of $d-1$ and with a general survival parameter $s$, i.e.~we investigate $\fm^*(d, \pi_{\text{sym}}, s)$ for $d \geq 2$. 

We tessellate $\Z^{d}$ with cubes $(Q'_x)_{x \in \Z^{d}}$ of size $3^{d}$. More precisely, for $x \in \Z^{d}$ we define
\begin{align*}
 Q'_x &= \{y \in \Z^{d} \colon \lVert y-3x \rVert_{\infty} \leq 1\}.
\intertext{Further, for technical reasons, for $a \in (\frac23, 1)$ we define}
 W_x &= \{y \in Q'_x \colon \lVert y-3x \rVert_1 \leq ad\},
\end{align*}
where $\lVert z \rVert_1 = \sum_{i=1}^{2d} \lvert z_i \rvert$ is the graph distance from $z \in \Z^d$ to $0$.  Informally, $W_x$ is the set of all vertices in $Q'_x$ which are ``sufficiently close'' to the centre of the cube.
Consider the box $Q'_x$ for some $x \in \Z^{d}$ and let $o \in W_x$. If there are frog paths in $Q_x'$ from $o$ to vertices close to the centres of all neighbouring boxes, i.e.~if the event 
\begin{equation*}
\bigcap_{e \in \mathcal{E}_d} \bigcup_{y \in W_{x + e}} \{o \fp{Q'_x} y\}
\end{equation*}
occurs, we call the vertex $o$ good. Note that this event only depends on the trajectories of all the frogs originating in the cube $Q'_x$ and the choice of $o$. If $o$ is good and is activated, then also the neighbouring cubes are visited. We show that the probability of a vertex being good is bounded from below uniformly in $d$ and this bound does not depend on the choice of $o$.

\begin{lemma}\label{lemma_recurrence_high_d_percolation_parameter_bound}
Consider the frog model $\fm^*(d,\pi_{\text{sym}},s)$. There are constants $\beta > 0$ and $d_0 \in \N$ such that for all $d \geq d_0$, $s > \frac34$, $\frac23 < a < 2- \frac{1}{s}$, $x \in \Z^d$ and $o \in W_x$
\begin{equation*}
\P(\text{$o$ is good}) > \beta.
\end{equation*}
\end{lemma}

To show this we first need to prove that many frogs in the cube are activated. In the proof of Theorem \ref{thm_d=2_arbitrary_weight_i} and Theorem \ref{thm_d>2_arbitrary_weight} this is done by means of Lemma~\ref{lemma_recurrence_cube_size} using the shape theorem. Here, we use a lemma that is analogous to Lemma~2.5 in \cite{AMP02pt}.

\begin{lemma}\label{lemma_recurrence_high_d_K_d}
 Consider the frog model $\fm^*(d,\pi_{\text{sym}},s)$. There exist constants $\gamma >0$, $\mu > 1$ and $d_0 \in \N$ such that for all $d \geq d_0$, $s > \frac34$, $\frac23 < a < 2- \frac{1}{s}$ and $o \in W_0$ we have
 \begin{equation*}
  \P\Bigl( \bigl\lvert \bigl\{y \in W_0 \colon o \fp{Q'_0} y \bigr\}\bigr\rvert \geq \mu^{\sqrt{d}}  \Bigr) \geq \gamma.
 \end{equation*}
\end{lemma}

\begin{proof}[Proof of Lemma \ref{lemma_recurrence_high_d_K_d}]
The proof consists of two parts. In the first part we show that with positive probability there are exponentially many vertices in $Q'_0$ reached from $o$ by frog paths in $Q'_0$, and in the second part we prove that many of these vertices are indeed in $W_0$. For the first part we closely follow the proof of Lemma~2.5 in \cite{AMP02pt} and rewrite the details for the convenience of the reader.

We examine the frog model with initially one active frog at $o$ and one sleeping frog at every other vertex in $Q'_0$ for $\sqrt{d}$ steps in time. Consider the sets $\mathcal{S}_0=\{o\}$ and $\mathcal{S}_k = \{x \in Q'_0 \colon \lVert x-o \rVert_1=k, \lVert x-o \rVert_{\infty}=1\}$ for $k \geq 1$ and let $\xi_k$ denote the set of active frogs which are in $\mathcal{S}_k$ at time $k$. We will show that, conditioned on an event to be defined later, the process $(\xi_k)_{k \in \N_0}$ dominates a process $(\tilde{\xi_k})_{k \in \N_0}$, which again itself dominates a supercritical branching process. The process $(\tilde{\xi_k})_{k \in \N_0}$ is defined as follows. Initially, there is one particle at $o$. Assume that the process has been constructed up to time $k \in \N_0$. In the next step each particle in $\tilde{\xi}_k$ survives with probability $s$. If it survives, it chooses one of the neighbouring vertices uniformly at random. If that vertex belongs to $\mathcal{S}_{k+1}$ and no other particle in $\tilde{\xi}_k$ intends to jump to this vertex, the particle moves there, activates the sleeping particle, and both particles enter $\tilde{\xi}_{k+1}$. Otherwise, the particle is deleted. In particular, if two or more particles attempt to jump to the same vertex, all of them will be deleted. Obviously, $\tilde{\xi}_k \subseteq \xi_k$ for all $k \in \N_0$. 

First, we show that for $d$ large it is unlikely that two particles in $\tilde{\xi}_k$ attempt to jump to the same vertex. To make this argument precise we need to introduce some notation. For $x \in \mathcal{S}_k$ and $y \in \mathcal{S}_{k+1}$ with $\lVert x-y \rVert_1=1$ define 
\begin{align*}
\mathcal{D}_x &= \{z \in \mathcal{S}_{k+1} \colon \lVert x-z\rVert_1 = 1\},\\
\mathcal{A}_y &= \{z \in \mathcal{S}_k \colon \lVert z-y \rVert_1 =1 \},\\
\mathcal{E}_x &= \{z \in \mathcal{S}_k \colon \mathcal{D}_x \cap \mathcal{D}_z \neq \emptyset \}.
\end{align*}
$\mathcal{D}_x$ denotes the set of possible descendants of $x$, $\mathcal{A}_y$ the set of ancestors of $y$ and $\mathcal{E}_x$ the set of enemies of $x$. Note that $\mathcal{E}_x = \bigcup_{y \in \mathcal{D}_x} (\mathcal{A}_y \setminus \{x\})$ is a disjoint union.
Let $n_x=\sum_{i=1}^d \1_{\{o_i=0,\, x_i\neq0\}}$. Then one can check that
\begin{align}\label{proof_recurrence_high_d_K_d_0}
 \lvert \mathcal{D}_x\rvert &= 2(d-\lVert o \rVert_1-n_x) + \lVert o \rVert_1 - (k-n_x) = 2d -\lVert o \rVert_1 - k - n_x,\\
 \lvert \mathcal{A}_y\rvert &= k+1. \nonumber
\end{align}
For $x \in \mathcal{S}_k$ let $\chi(x)$ denote the number of particles of $\tilde{\xi}_k$ in $x$. Note that $\chi(x) \in \{0,2\}$ for any $x \in \mathcal{S}_k$ with $k \in \N$.

Let $\zeta_{xy}^k$ denote the indicator function of the event that there is $z \in \mathcal{E}_x$ with $\chi(z)\geq 1$ such that one of the particles at $z$ intends to jump to $y$ at time $k+1$. If $\zeta_{xy}^k=1$, then a particle on $x$ cannot move to $y$ at time $k+1$.

Further, we introduce the event $U_x= \{\chi(z)=2 \text{ for all } z \in \mathcal{E}_x\}$. This event describes the worst case for $x$, when it is most likely that particles at $x$ will not be able to jump.
For $k \leq \sqrt{d}$ we have
\begin{equation*}
 \P(\zeta_{xy}^k=1) \leq \P(\zeta_{xy}^k=1 \mid U_x) \leq \sum_{z \in \mathcal{A}_y \setminus \{x\}} \frac{2s}{2d} = \frac{ks}{d} \leq \frac{1}{\sqrt{d}}.
\end{equation*}
Given $\sigma > 0$ we choose $d$ large such that $\P(\zeta_{xy}^k=1) < \sigma$ for all $k \leq \sqrt{d}$.
Now, we consider the set of all descendants $y$ of $x$ such that there is a particle at some vertex $z \in \mathcal{E}_x$ that tries to jump to $y$ at time $k+1$. This set contains $\sum_{y \in \mathcal{D}_x} \zeta_{xy}^k$ elements. Let $\zeta_x^k$ denote the indicator function of the event $\bigl\{\sum_{y \in \mathcal{D}_x} \zeta_{xy}^k > 2\sigma d\bigr\}$. If $\zeta_{x}^k=1$, then more than $2\sigma d$ of the $2d$ neighbours of $x$ are blocked to a particle at $x$.

The random variables $\{\zeta_{xy}^k \colon y \in \mathcal{D}_x\}$ are independent with respect to $\P(\cdot \mid U_x)$ as $\mathcal{E}_x = \bigcup_{y \in \mathcal{D}_x} (\mathcal{A}_y \setminus \{x\})$ is a disjoint union. Using $2d-ad-2k \leq \lvert\mathcal{D}_x\rvert \leq 2d$ and a standard large deviation estimate we get for $k \leq \sqrt{d}$
\begin{align*}
 \P(\zeta_x^k=1) &\leq \P\biggl(\sum_{y \in \mathcal{D}_x} \zeta_{xy}^k > 2\sigma d \Bigm\vert U_x\biggr) \\
                 &\leq \P\biggl(\frac{1}{\lvert \mathcal{D}_x\rvert} \sum_{y \in \mathcal{D}_x} \zeta_{xy}^k > \sigma \Bigm\vert U_x \biggr)\\
                 &\leq \e^{-c_1\lvert\mathcal{D}_x\rvert} \\
                 &\leq \e^{-c_2 d}
\end{align*}
with constants $c_1, c_2 > 0$. Next, let us consider the bad event
\begin{equation*}
 B = \bigcup_{k=1}^{\sqrt{d}} \bigcup_{x \in \tilde{\xi}_k} \{\zeta_x^k=1\}.
\end{equation*}
Then with $\lvert\tilde{\xi}_k\rvert \leq 2^k \leq 2^{\sqrt{d}}$ we get
\begin{equation*}
 \P(B) \leq \sqrt{d} \cdot 2^{\sqrt{d}} \cdot \e^{-c_2 d}.
\end{equation*}
In particular $\P(B)$ can be made arbitrarily small for $d$ large.
Conditioned on $B^c$, in each step for every particle there are at least 
\begin{equation*}
\lvert\mathcal{D}_x\rvert-2 \sigma d -1 \geq (2-a-2\sigma)d - 3 \sqrt{d} 
\end{equation*}
available vertices in $\mathcal{S}_{k+1}$, i.e.~vertices a particle at $x$ can jump to in the next step. Thus, conditioned on $B^c$, the process $\tilde{\xi}_k$ dominates a branching process with mean offspring at least
\begin{equation*}
 \frac{\bigl((2-a-2\sigma)d - 3 \sqrt{d}\bigr) \cdot 2 \cdot s}{2d}.
\end{equation*}
For $\sigma$ small and $d$ large the mean offspring is bigger than $1$ as we assumed $a < 2-\frac{1}{s}$. Since a supercritical branching process grows exponentially with positive probability, there are constants $c_3 >1$, $q \in (0,1)$ that do not depend on $d$ such that
\begin{equation}\label{proof_recurrence_high_d_K_d}
\P\bigl( \lvert\tilde{\xi}_{\sqrt{d}}\rvert \geq c_3^{\sqrt{d}}\bigr) \geq q.
\end{equation}
For the second part of the proof condition on the event $\bigr\{\lvert\tilde{\xi}_{\sqrt{d}}\rvert \geq c_3^{\sqrt{d}}\bigl\}$ and choose $0 < \varepsilon <a-\frac23$. If $\lVert o \rVert_1 \leq (a-\varepsilon)d$, all particles of $\tilde{\xi}_{\sqrt{d}}$ are in $W_0$ for $d$ large. This immediately implies the claim of the lemma. Otherwise, let $n=\lvert\tilde{\xi}_{\sqrt{d}}\rvert$, enumerate the particles in $\tilde{\xi}_{\sqrt{d}}$ and let $\tilde{S}^i$, $1 \leq i \leq n$, denote the position of the $i$-th particle. Further, we define for $1 \leq i \leq n$
\begin{equation*}
 X_i =
 \begin{cases}
  1 & \text{if $\lVert \tilde{S}^i \rVert_1 \leq \lVert o \rVert_1 $}, \\
  0 & \text{otherwise.}
 \end{cases}
\end{equation*}
It suffices to show that $\P(X_1=1)>0$. Then Lemma~\ref{lemma_sum_rv} applied to the random variables $X_1, \ldots, X_n$ implies that with positive probability a positive proportion of the particles in $\tilde{\xi}_{\sqrt{d}}$ indeed have $L_1$-norm smaller than $o$, and are thus in $W_0$. Together with \eqref{proof_recurrence_high_d_K_d} this finishes the proof.

For the proof of the claim let $\tilde{S}^1_k$ denote the position of the ancestor of $\tilde{S}^1$ in $\mathcal{S}_k$, where $0 \leq k \leq \sqrt{d}$. Note that $\tilde{S}^1_0 = o$ and $\tilde{S}^1_{\sqrt{d}} = \tilde{S}^1$.

We are interested in the process $(\lVert \tilde{S}_k^1 \rVert_1)_{1 \leq k \leq \sqrt{d}}$. By the construction of the process $(\tilde{\xi}_k)_{k \in \N_0}$ it either increases or decreases by $1$ in every step. The positions $\tilde{S}_k^1$ and $\tilde{S}_{k+1}^1$ differ in exactly one coordinate. If this coordinate is changed from $0$ to $\pm 1$, then $\lVert \tilde{S}_{k+1}^1\rVert_1$ = $\lVert \tilde{S}_k^1 \rVert_1 +1$. If it is changed from $\pm 1$ to $0$, then we have $\lVert \tilde{S}_{k+1}^1\rVert_1$ = $\lVert \tilde{S}_k^1 \rVert_1 -1$. There are at least $(a-\varepsilon)d-\sqrt{d}$ many $\pm 1$-coordinates in $\tilde{S}_k^1$ that can be changed to $0$. As we also know that $\tilde{S}_{k+1}^1 \in \mathcal{D}_{\tilde{S}_k^1}$, we have for all $k \leq \sqrt{d}$ by \eqref{proof_recurrence_high_d_K_d_0} and the choice of $\varepsilon$
\begin{equation*}
 \P\bigl(\lVert \tilde{S}_{k+1}^1\rVert_1 = \lVert \tilde{S}_k^1 \rVert_1 -1\bigr) 
	\geq \frac{(a-\varepsilon)d-\sqrt{d}}{\lvert\mathcal{D}_{\tilde{S}_k^1}\rvert} 
	\geq \frac{(a-\varepsilon)d - \sqrt{d}}{2d - (a-\varepsilon)d} 
	> \frac12
\end{equation*}
for $d$ large. Hence, $\lVert \tilde{S}_k^1 \rVert_1$ dominates a random walk with drift on $\Z$ started in $\lVert o \rVert_1$. Therefore, 
\begin{equation*}
\P(X_1 = 1) = \P\bigl(\lVert \tilde{S}_{\sqrt{d}}^1 \rVert_1 \leq \lVert o \rVert_1\bigr) \geq \frac12,
\end{equation*}
which finishes the proof.
\end{proof}

\begin{proof}[Proof of Lemma~\ref{lemma_recurrence_high_d_percolation_parameter_bound}]
By Lemma~\ref{lemma_recurrence_high_d_K_d}, with probability at least $\gamma$ there are frog paths in $Q'_x$ from $o$ to at least $\mu^{\sqrt{d}}$ vertices in $W_x$ for $d$ large. We divide the frogs on these vertices into $2d$ groups of size at least $\mu^{\sqrt{d}}/2d$ and assign each group the task of visiting one of the neighbouring boxes $W_{x+e}$, $e \in \mathcal{E}_d$. Notice that this job is done if at least one of the frogs in the group visits at least one vertex in the neighbouring box. If all groups succeed, $o$ is good. Any frog in any group is just three steps away from its respective neighbouring box $W_{x+e}$, $e \in \mathcal{E}_d$, and thus has probability at least $(\frac{s}{2d})^3$ of achieving its group's goal. Hence,
\begin{equation*}
\P(\text{$o$ is good}) \geq  \Bigl(1- \Bigl(1-\Bigl(\frac{s}{2d}\Bigr)^3\Bigr)^{\mu^{\sqrt d}/{2d}} \Bigr)^{2d} \gamma 
                       \geq \frac{\gamma}{2}
\end{equation*}
for $d$ large.  
\end{proof}

In the other recurrence proofs we couple the frog model with percolation by calling a cube open if its centre is good. Here, the choice of a ``starting'' vertex, like the centre, is not independent of the other cubes. Therefore, we cannot directly couple the frog model with independent percolation. However, the following lemma allows us to compare the distributions of a frog cluster and a percolation cluster.

\begin{lemma}\label{lemma_recurrence_high_d_fc=c}
Consider the frog model $\fm^*(d,\pi_{\text{sym}},s)$. Let $\beta >0$ and assume that $\P(\text{$o$ is good}) > \beta$ for all $o \in W_x$, $x \in \Z^d$. Further, consider independent site percolation on $\Z^d$ with parameter $\beta$. Then for all sets $A \subseteq \Z^d$, $v \in \Z^d$ and for all $k \geq 0$
\begin{equation*}
 \P(\lvert A \cap C_v\rvert \geq k) \leq \P\Bigl(\Bigl\lvert \bigcup_{x \in A}Q'_x\cap \fc_{3v}^*\Bigr\rvert \geq k\Bigr).
\end{equation*}
\end{lemma}

\begin{proof}
For technical reasons we introduce a family of independent Bernoulli random variables $(X_o)_{o \in \Z^d}$ which are also independent of the choice of all the trajectories of the frogs and satisfy $\P(X_o=1) = \P(\text{$o$ is good})^{-1}\beta$. Their job will be justified soon. Further, we fix an ordering of all vertices in $\Z^d$.

Now we are ready to describe a process that explores a subset of the frog cluster $\fc_{3v}^*$. Its distribution can be related to the cluster $C_v$ in independent site percolation with parameter $\beta$. The process is a random sequence $(R_t, D_t, U_t)_{t\in \N_0}$ of tripartitions of $\Z^d$. As the letters indicate, $R_t$ will contain all sites reached by time $t$, $D_t$ all those declared dead by time $t$, and $U_t$ the unexplored sites. We construct the process in such a way that for all $t \in \N_0$, $x \in R_t$ and $e \in \mathcal{E}_d$ there is $y \in W_{x +e}$ such that there is a frog path from $3v$ to $y$ in $\bigcup_{x \in R_t}Q'_x$. We start with $R_0 = D_0 = \emptyset$ and $U_0 = \Z^d$. If $3v$ is good and $X_{3v}=1$, set $U_1 = \Z^d \setminus \{v\}$, $R_1=\{v\}$, and $D_1=\emptyset$. Otherwise, stop the algorithm. If the process is stopped at time $t$, let $U_s = U_{t-1}$, $R_s = R_{t-1}$ and $D_s = D_{t-1}$ for all $s \geq t$. Assume we have constructed the process up to time $t$. Consider the set of all sites in $U_t$ that have a neighbour in $R_t$. If it is empty, stop the process. Otherwise, pick the site $x$ in this set with the smallest number in our ordering. By the choice of $x$ there is $y \in W_x$ such that there is a frog path from $3v$ to $y$ in $\bigcup_{z \in R_t} Q'_z$. Choose any vertex $y$ with this property. If $y$ is good and $X_y = 1$, set 
\begin{equation*}
R_{t+1} = R_t \cup \{x\},\ D_{t+1} = D_t, \ U_{t+1}=U_t \setminus \{x\}. 
\end{equation*}
Otherwise, update the sets as follows:
\begin{equation*}
R_{t+1} = R_t,\ D_{t+1} = D_t \cup \{x\}, \ U_{t+1}=U_t \setminus \{x\}
\end{equation*}
In every step $t$ the algorithm picks an unexplored site $x$ and declares it to be reached or dead, i.e.~added to the set $R_{t}$ or $D_t$. The probability that $x$ is added to $R_t$ equals $\beta$. This event is (stochastically) independent of everything that happened before time $t$ in the algorithm. Note that every unexplored neighbour of a reached site will eventually be explored due to the fixed ordering of all sites.

In the same way we can explore independent site percolation on $\Z^d$ with parameter $\beta$. Construct a sequence $(R_t', D_t', U_t')_{t\in \N_0}$ of tripartitions of $\Z^d$ as above, but whenever the algorithm evaluates whether a site $x$ is declared reached or dead we toss a coin independently of everything else. Note that $\bigcup_{t \in \N_0} R_t' = C_v$, where $C_v$ is the cluster containing $v$. This exploration process is well known for percolation, see e.g.~\cite[Proof of Theorem 4, Chapter 1]{BR06}.

By construction, $\bigcup_{t\in \N_0} R_t$ equals the percolation cluster $C_v$ in distribution. The claim follows since for every $x \in \bigcup_{t\in \N_0} R_t$ there is a $y \in W_x$ such that there is a frog path from $3v$ to $y$, i.e.~$y \in \fc_{3v}^*$. 
\end{proof}

Now we can show Proposition~\ref{prop_d>2_arbitrary_drift_large_d}. Note that we are again working with the frog model $\fm(d,\pi_{w,\alpha})$ (without death).

\begin{proof}[Proof of Proposition~\ref{prop_d>2_arbitrary_drift_large_d}]
Throughout this proof we assume that $d$ is so large that Lemma~\ref{lemma_recurrence_high_d_percolation_parameter_bound} is applicable for $d-1$ and $p_c(d-1) < \beta$, where $\beta$ is the constant introduced in the statement of Lemma~\ref{lemma_recurrence_high_d_percolation_parameter_bound}. This is possible because of Lemma~\ref{lemma_pc_high_d}. These assumptions in particular imply that we can use Lemma~\ref{lemma_recurrence_high_d_fc=c} and that the percolation introduced there is supercritical.

Consider the sequence of hyperplanes $(H_{-n})_{n \in \N_0}$ defined in \eqref{definition_hyperplane} and let $A$ denote the event that there is at least one frog $v_n$ activated in every hyperplane $H_{-n}$. For technical reasons we want $v_n$ of the form $v_n = (-n, 3w_n)$ for some $w_n \in \Z^{d-1}$. We first show that $A$ occurs with positive probability. To see this consider the first hyperplane $H_0$ and couple the frogs in this hyperplane with $\fm^*(d-1, \pi_{\text{sym}}, 1-w)$ in the following way: Whenever a frog takes a step in $\pm e_1$-direction, i.e.~leaves its hyperplane, it dies instead. By \cite[Theorem 1.8]{AMP02pt} (or Lemma~\ref{lemma_recurrence_high_d_fc=c}) this process survives with positive probability if $w$ is sufficiently small (independent of the dimension $d$). This means that infinitely many frogs are activated in $H_0$. Obviously, this implies the claim. 

From now on we condition on the event $A$. Note that $\fc_{v_n} \subseteq \fc_0$ for $n\in\N$.
Analogously to the proofs in the last sections we introduce boxes 
\begin{equation*}
\fb_n' = \{-n\} \times [-(3\sqrt{n}+1), 3\sqrt{n}+1]^{d-1} 
\end{equation*}
for $n \in \N$. 
We claim that analogously to Lemma~\ref{percolation_density} there are constants $a, b>0$ and $N \in \N$ such that for $n \geq N$
\begin{equation} \label{proof_thm_high_d_1}
 \P\bigl(\lvert\fb_n' \cap \fc_0\rvert \geq a n^{(d-1)/2}\bigr) \geq b.
\end{equation}
To prove this claim let $a,b>0$ and $N \in \N$ be the constants provided by Lemma \ref{percolation_density} for percolation with parameter $\beta$. For $n \geq N$ couple the frog model with $\fm^*(d-1, \pi_{\text{sym}}, 1-w)$ in the hyperplane $H_n$ as above. Let $B_n' = [-\sqrt{n}, \sqrt{n}]^{d-1}$ and note that $B_n'$ corresponds to $\fb_n'$ restricted to $H_n$ after rescaling. Then by Lemma~\ref{lemma_recurrence_high_d_fc=c} and Lemma~\ref{percolation_density}
\begin{align*}
 \P \bigl(\lvert\fb_n' \cap \fc_{v_n}\rvert \geq a n^{(d-1)/2 } | A \bigr)
    &\geq \P \bigl(\lvert\fb_n' \cap (\{-n\} \times \fc_{3w_n}^*)\rvert \geq a n^{(d-1)/2) } | A \bigr) \\
    &\geq \P\bigl(\lvert B_n' \cap C_{w_n}\rvert \geq a n^{(d-1)/2)}  | A \bigr) \\
    &\geq b.
\end{align*}
Here, $C_{w_n}$ is the open cluster containing $w_n$ in a percolation model with parameter $\beta$ in $\Z^{d-1}$, independently of the frogs.
As $\fc_{v_n} \subseteq \fc_0$, this implies inequality~\eqref{proof_thm_high_d_1}.

By Lemma~\ref{lemma_hitting_probability_RW_drift} and \eqref{proof_thm_high_d_1}, the probability that there is at least one activated frog in $\fb_n'$ that reaches $0$ is at least
\begin{equation*}
 \Bigl(1-(1-c'n^{-(d-1)/2})^{an^{(d-1)/2}}\Bigr)b \geq \bigl(1 - \e^{-ac'}\bigr)b,
\end{equation*}
where $c'>0$ is a constant. Altogether we get by Lemma~\ref{lemma_sum_rv}
\begin{align*}
 \P(\text{$0$ visited infinitely often}) &=    \lim_{n \to \infty} \P(\text{$0$ is visited $\varepsilon n$ many times }) \\
                                         &\geq \lim_{n \to \infty} \P\biggl( \sum_{i=1}^n \1_{\{\exists x \in \fb_n' \cap \fc_{0} \colon x \to 0 \}} \geq \varepsilon n \biggr)\\
                                         &\geq \Bigl(\bigl(1 - \e^{-ac'}\bigr)b - \varepsilon \Bigr) > 0
\end{align*}
for $\varepsilon$ sufficiently small. The claim now follows from Theorem~\ref{lemma_zero_one_law}.
\end{proof}


To prove Proposition~\ref{prop_d>2_arbitrary_drift_small_d} we again first study the frog model with death $\fm^*(d, \pi_{\text{sym}},s)$ in the hyperplanes and couple it with percolation. This time we use cubes of size $(2K+1)^{d}$ for some $K \in \N_0$. By choosing $K$ large we increase the number of frogs in the cubes. In the proof of the previous proposition this was done by increasing the dimension $d$. For $x \in \Z^d$ and $K \in \N_0$ we define
\begin{align*}
 q_x &= q_x(K) = (2K+1)x, \\
 Q_x &= Q_x(K) = \{y \in \Z^{d} \colon \lVert y-q_x \rVert_{\infty} \leq K\}.
\end{align*}

Note that this definition coincides with \eqref{def_box}.
In analogy to Lemma~\ref{lemma_recurrence_high_d_fc=c} the frog cluster dominates a percolation cluster.

\begin{lemma}\label{lemma_frog_model_with_death_percolation_arbitrary_d}
For $d \geq 2$ consider the frog model $\fm^*(d,\pi_{\text{sym}},s)$ and supercritical site percolation on $\Z^d$. There are constants $s_r(d) < 1$ and $K \in \N_0$ such that for any $s \geq s_r(d)$, $A \subseteq \Z^d$, $v \in \Z^d$ and for all $k \geq 0$
\begin{equation*}
 \P(\lvert A \cap C_v \rvert \geq k) \leq \P\Bigl(\Bigl\lvert \bigcup_{x \in A}Q_x\cap \fc_{q_v}^*\Bigr\rvert \geq k\Bigr).
\end{equation*}
\end{lemma}

\begin{proof}
We couple the frog model with percolation as follows: A site $x \in \Z^{d}$ is called open if for every $e \in \mathcal{E}_{d}$ there exists a frog path from $q_x$ to $q_{x + e}$ in $Q_x$. Note that $x \in C_v$ now implies $q_x \in \fc_{q_v}^*$ for any $v \in \Z^d$. We denote the probability of a site $x$ to be open by $p(K,s)$. By Lemma~\ref{lemma_recurrence_cube_size} $p(K,1)$ is close to $1$ for $K$ large. As in the proof of Lemma~\ref{lemma_recurrence_small_drift} one can show that $\lim_{s \to 1} p(K,s)= p(K,1)$. Thus, we can choose $K \in \N$ and $s_r >0$ such that $p(K,s) > p_c(d)$ for all $s > s_r$, i.e. the percolation is supercritical. 
\end{proof}

\begin{proof}[Proof of Proposition~\ref{prop_d>2_arbitrary_drift_small_d}]
 Using Lemma~\ref{lemma_frog_model_with_death_percolation_arbitrary_d} instead of Lemma~\ref{lemma_recurrence_high_d_fc=c} and boxes $Q_x$ instead of $Q_x'$, the proof is analogous to the proof of Proposition~\ref{prop_d>2_arbitrary_drift_large_d}.
\end{proof}

\begin{proof}[Proof of Theorem \ref{thm_d>2_arbitrary_drift_i}]
 Theorem \ref{thm_d>2_arbitrary_drift_i} follows from Proposition~\ref{prop_d>2_arbitrary_drift_large_d} and Proposition~\ref{prop_d>2_arbitrary_drift_small_d}.
\end{proof}

%% file: frogsZd_transience_dgeq2_arbitrary_drift.tex

\subsection*{Transience for $d\geq 2$ and arbitrary drift}

\begin{proof}[Proof of Theorem~\ref{thm_d=2_arbitrary_drift_ii} and Theorem~\ref{thm_d>2_arbitrary_drift_ii}]
Let the parameters $\alpha>0$ and $d\geq 2$ be fixed throughout the proof. 
For $x \in \Z^d$ we define 
\begin{equation}
 L_x = \{y \in \Z^d \colon y_i = x_i \text{ for all $2 \leq i \leq d$}\}.
\end{equation}
$L_x$ consists of all vertices which agree in all coordinates with $x$ except the $e_1$-coordinate. 
The key observation used in the proof is that all particles mainly move along these lines if the weight $w$ is large.

We dominate the frog model by a branching random walk on $\Z^d$. At time $n=0$ the branching random walk starts with one particle at the origin. At every step in time every particle produces offspring as follows: For every particle located at $x \in \Z^d$ consider an independent copy of the frog model. At any vertex $z \in \Z^d \setminus L_x$ the particle produces $|\{y \in L_x \colon x \fp{L_x} y, y \to z\}|$ many children. Notice that this number might be $0$ or infinite. The particle does not produce any offspring at a vertex in $L_x$. 
Further, note that the particles reproduce independently of each other as we use independent copies of the frog model to generate the offspring.

One can couple this branching random walk with the original frog model.
To explain the coupling, let us briefly describe how to go from the original frog model to the branching random walk. Recall that the frog model is entirely determined by a set of trajectories $(S_n^x)_{n \in \N_0, x \in \Z^d}$ of random walks. We use this set of trajectories to produce the particles in the first generation of the branching random walk, i.e.~the children of the particle initially at $0$, as explained above. Now, assume that the first $n$ generations of the branching random walk have been created. Enumerate the particles in the $n$-th generation. When generating the offspring of the $i$-th particle in this generation, delete all trajectories of the frog model used for generating the offspring of a particle~$j$ with $j < i$ or a particle in an earlier generation, and replace them by independent trajectories. Otherwise, use the original trajectories.

One can check that the branching random walk dominates the frog model in the following sense: 
For every frog in $\Z^d \setminus L_0$ that is activated and visits $0$ there is a particle at $0$ in the branching random walk. Thus, the number of visits to the origin by particles in the branching random walk is at least as big as the number of visits to $0$ by frogs in the frog model, not counting those visits to $0$ made by frogs initially in $L_0$. 
Note that, if the frog model was recurrent, then almost surely there would be infinitely many frogs in $\Z^d \setminus L_0$ activated that return to $0$. In particular, also in the branching random walk infinitely many particles would return to $0$. Therefore, to prove transience of the frog model it suffices to show that in the branching random walk only finitely many particles return to $0$ almost surely.

Let $D_n$ denote the set of descendants in the $n$-th generation of the branching random walk. Further, for $i \in D_n$ let $X_n^i$ be the $e_1$-coordinate of the location of particle $i$. Define for $\theta >0$ and $n \in \N_0$
\begin{equation}
 \mu = \E \Bigl[ \sum_{i \in D_1} \e^{-\theta  X_1^i} \Bigr] \qquad \text{and}  \qquad  M_n = \frac{1}{\mu^n} \sum_{i \in D_n} \e^{-\theta X_n^i}.
\end{equation}
We claim that $\mu <1$ for $w$ close to $1$ and $\theta$ small, which, in particular, implies that $(M_n)_{n \in \N_0}$ is well-defined. We show this claim in the end of the proof. We next show that $(M_n)_{n \in \N_0}$ is a martingale with respect to the filtration $(\mathcal{F}_n)_{n \in \N_0}$ with $\mathcal{F}_n=\sigma \bigl(D_1, \ldots, D_{n}, (X^i_1)_{i \in D_1}, \ldots, (X^i_{n})_{i \in D_{n}} \bigr)$.

Obviously, $M_n$ is $\mathcal{F}_n$-measurable. For a particle $i \in D_n$ denote its descendants in generation $n+1$ by $D_{n+1}^i$. Since particles branch independently, we get
\begin{align*}
\E[ M_{n+1} | \mathcal{F}_n ] &= \E \Bigl[\frac{1}{\mu^{n+1}} \sum_{i \in D_{n+1}} \e^{-\theta X_{n+1}^i} \bigm\vert \mathcal{F}_n \Bigr] \\
                              &= \frac{1}{\mu^n} \sum_{i \in D_{n}} \e^{-\theta X_{n}^i} \cdot \frac{1}{\mu} \E \Bigr[ \sum_{j \in D_{n+1}^i} \e^{-\theta \left( X_{n+1}^j - X_{n}^i \right)} \Bigl].
\end{align*}
Note that the expectation on the right hand side is independent of $i$ and $n$ and therefore, by the definition of $\mu$, we conclude
\begin{align*}
\E[ M_{n+1} | \mathcal{F}_n ] = M_n.
\end{align*}
This calculation also yields $\E[\lvert M_n\rvert]= \E[M_n]=\E[M_0]=1$, and therefore $M_n \in \mathcal{L}^1$. This in particular implies that $M_n$ is finite almost surely for every $n \in \N_0$. Thus, $X_n^i=0$ can only occur for finitely many $i \in D_n$ almost surely for every $n \in \N_0$, i.e.~in every generation only finitely many particles can be at $0$.
By the martingale convergence theorem, there exists an almost surely finite random variable $M_\infty$, such that $\lim_{n \to \infty} M_n = M_\infty$ almost surely.
Combining this with $\mu <1$, we get $\lim_{n \to \infty}\sum_{i \in D_{n}} \e^{-\theta X_{n}^i} = 0$ almost surely. Hence, $X_n^i=0$ for some $i \in D_n$ occurs only for finitely many times $n$. Overall, this shows that the branching random walk is transient.

It remains to show $\mu < 1$. Note that the particles in $D_1$ are at vertices in the set $\{y \in \Z^d\setminus L_0 \colon 0 \fp{L_0} y\}$. Therefore, for the calculation of $\mu$ we first need to consider all sites in $L_0$ that are reached from $0$ by frog paths in $L_0$. The idea is to control the number of frogs activated on the negative $e_1$-axis using Lemma~\ref{lemma_1d_fm} and estimating the number of frogs activated on the positive $e_1$-axis by assuming the worst case scenario that all of them will be activated. Then, for every $k \in \Z$ we have to estimate the number of vertices with $e_1$-coordinate $k$ visited by each of these active frogs on the $e_1$-axis. Due to the definition of $\mu$, the sites visited by frogs on the positive $e_1$-axis do not contribute much to $\mu$. 
Recall that $H_k$ denotes the hyperplane that consists of all vertices with $e_1$-coordinate equal to $k\in \Z$, see \eqref{definition_hyperplane}. For $k,i \in \Z$ define 
\begin{equation*}
N_{k,i} = \lvert\{x \in H_k \setminus L_0 \colon (i,0, \ldots, 0) \to x\}\rvert. 
\end{equation*}
As $N_{k,i}$ equals $N_{k-i,0}$ in distribution for all $i,k \in \Z$, we get 
\begin{align} \label{proof_transience_arbirtrary_drift_1}
  \mu &= \E \Bigl[ \sum_{i \in D_1} \e^{-\theta  X_1^i} \Bigr] \nonumber \\
      &= \sum_{i=-\infty}^{\infty} \sum_{k=-\infty}^{\infty} \P\bigl(0 \fp{L_0} (i,0, \ldots, 0)\bigr) \E[N_{k,i}] \e^{-\theta k} \nonumber\\
      &= \sum_{k=-\infty}^{\infty}  \E[N_{k,0}] \e^{-\theta k} \sum_{i=-\infty}^{\infty} \e^{-\theta i} \P\bigl(0 \fp{L_0} (i,0, \ldots, 0)\bigr).
\end{align}
Note that $\P\bigl(0 \fp{L_0} (i,0, \ldots, 0)\bigr)$ is smaller or equal than the probability of the event $\{0 \fp{\Z} i\}$ in the frog model $\fm(1,1,\alpha)$. Hence, by Lemma~\ref{lemma_1d_fm}, there is a constant $c_1 >0$ such that $\P\bigl(0 \fp{L_0} (i,0, \ldots, 0)\bigr) \leq \e^{c_1i}$ for all $i \leq 0$. Thus, \eqref{proof_transience_arbirtrary_drift_1} implies that for $\theta<c_1$ there is a constant $c_2=c_2(\theta)< \infty$ such that
\begin{equation}\label{proof_transience_arbirtrary_drift_2}
 \mu \leq c_2 \sum_{k=-\infty}^{\infty}  \E[N_{k,0}] \e^{-\theta k}.
\end{equation}
Next, we estimate $\E[N_{k,0}]$, the expected number of vertices in $H_k \setminus L_0$ visited by a single particle starting at $0$. Recall that the trajectory of frog $0$ is denoted by $(S_n^0)_{n\in \N_0}$. We define $T_k = \min\{n \in \N_0 \colon S_n^0 \in H_k\}$, the entrance time of the hyperplane $H_k$, and $T_k' = \max\{n \in \N_0 \colon S_n^0 \in H_k\}$, the last time frog $0$ is in the hyperplane $H_k$. Obviously, $N_{k,0}=0$ on the event $\{T_k = \infty\}$. Hence, assume we are on $\{T_k < \infty\}$. The particle can only visit a vertex in $H_k \setminus L_0$ at time $T_k$ if the random walk took at least one step in non-$e_1$-direction up to time $T_k$. This happens with probability $\E[1-w^{T_k}]$. Furthermore, the number of vertices visited in $H_k$ after time $T_k$ can be estimated by the number of steps in non-$e_1$-direction taken between times $T_k$ and $T_k'$. This number is binomially distributed and, thus, its expectation equals $(1-w)\E[T_k'-T_k]$. Overall, this implies
\begin{align*}
 \E[N_{k,0}] \leq  \P(T_k < \infty) \bigr(\E\bigl[1- w^{T_k} \mid T_k < \infty \bigr] + (1-w) \E\bigl[T_k' - T_k \mid T_k <\infty\bigr] \bigl).
\end{align*}
For $k < 0$ the probability $\P(T_k < \infty)$ decays exponentially in $k$ by Lemma~\ref{lemma_hitting_probability_hyperplane}. Therefore, we can choose $\theta$ small such that $\P(T_k < \infty) \e^{-\theta k} \leq  \e^{-\theta \lvert k \rvert}$ for all $k \in \Z$. Thus, \eqref{proof_transience_arbirtrary_drift_2} implies
\begin{equation}\label{proof_transience_arbirtrary_drift_3}
 \mu \leq  c_2 \sum_{k=-\infty}^{\infty}  \e^{-\theta \lvert k \rvert}\Bigr( \E\bigl[1- w^{T_k} \mid T_k < \infty \bigr] + (1-w) \E\bigl[T_k' - T_k \mid T_k <\infty\bigr] \Bigl).
\end{equation}
Note that the sum in \eqref{proof_transience_arbirtrary_drift_3} is finite as $\E\bigl[T_k' - T_k \mid T_k <\infty\bigr]$ is independent of $k$.
By monotone convergence $\lim_{w \to 1} \mu =0$ and the right hand side of \eqref{proof_transience_arbirtrary_drift_3} is continuous in $w$. Therefore, we can choose $w$ close to $1$ such that $\mu < 1$, as claimed.
\end{proof}

%% file: frogsZd_transience_d2_arbitrary_weight.tex

\subsection*{Transience for $d=2$ and arbitrary weight}

\begin{proof}[Proof of Theorem~\ref{thm_d=2_arbitrary_weight_ii}]
Let $w >0$ be fixed throughout the proof. As in the proof of Theorem~\ref{thm_d=2_arbitrary_drift_ii} and Theorem~\ref{thm_d>2_arbitrary_drift_ii} we dominate the frog model by a branching random walk. This time we use a one-dimensional branching random walk on $\Z$. For the construction of the process, let $\xi$ be the number of activated frogs in an independent one-dimensional frog model $\fm^*(1,\pi_{\text{sym}}, 1-w)$ with two active frogs at $0$ initially. At time $n=0$, the branching random walk starts with one particle in the origin. At every time $n \in \N$, the process repeats the following two steps. First, every particle produces offspring independently of all other particles with the number of offspring being distributed as $\xi$. Then, each particle jumps to the right with probability $\frac{1+\alpha}{2}$ and to the left with probability $\frac{1-\alpha}{2}$. 

As an intermediate step to understand the relation between the frog model and this branching random walk on $\Z$, we first couple the frog model with a branching random walk on $\Z^2$ with initially one particle at $0$. Partition the lattice $\Z^2$ into hyperplanes $(H_n)_{n \in \Z}$ as defined in \eqref{definition_hyperplane}. Let the frog model $\fm(2,\pi_{w,\alpha})$ with initially two active frogs at $0 \in H_0$ evolve and stop every frog when it first enters $H_1$ or $H_{-1}$. Every frog leaves its hyperplane in every step with probability $w$. Thus, the number of stopped frogs is distributed according to $\xi$. A stopped frog is in $H_1$ with probability $\frac{1+\alpha}{2}$ and in $H_{-1}$ with probability $\frac{1-\alpha}{2}$. The stopped particles form the offspring of the particle at $0$ in the branching random walk. We repeat this procedure to generate the offspring of an arbitrary particle in the branching random walk. Introduce an ordering of all particles in the branching random walk and let the particles branch one after another. Before generating the offspring of the $i$-th particle, refill every vertex which is no longer occupied by a sleeping frog with an extra independent sleeping frog. Unstop frog $i$ and let it continue its work as usual, ignoring all other stopped frogs. Note that there is a sleeping frog at the starting vertex of frog $i$ that is immediately activated. This explains our definition of $\xi$. 
Again stop every frog once it enters one of the neighbouring hyperplanes. These newly stopped frogs form the offspring of the $i$-th particle. This procedure creates a branching random walk with independent identically distributed offspring. Every vertex visited in the frog model is obviously also visited by the branching random walk. 

Now, project all particles in the intermediate two-dimensional branching random walk onto the first coordinate. This creates a branching random walk on $\Z$ distributed as the one described above. The construction shows that transience of this one-dimensional branching random walk implies transience of the frog model.

To prove that the one-dimensional branching random walk is transient for $\alpha$ close to $1$, we proceed as in the proof of Theorem~\ref{thm_d=2_arbitrary_drift_ii} and Theorem~\ref{thm_d>2_arbitrary_drift_ii}. The proof only differs in the calculation of the parameter $\mu$ defined by  
\begin{equation*}
 \mu = \E \Bigl[ \sum_{i \in D_1} \e^{-\theta  X_1^i} \Bigr] 
\end{equation*}
for $\theta >0$ with $D_1$ denoting the set of descendants in the first generation of the branching random walk and $X_1^i$ the $e_1$-coordinate of the location of particle $i \in D_1$. Here, we immediately get
\begin{equation*}
 \mu = \frac12  \bigl((1-\alpha) \e^{\theta} + (1+\alpha) \e^{-\theta}\bigr) \E[\xi].
\end{equation*}
Lemma~\ref{lemma_1d_fm} implies $\E[\xi] < \infty$. Thus, we can choose $\theta = \log\bigl(2\E[\xi]\bigr)$. Then $\lim_{\alpha \to 1} \mu = \frac12$ and by continuity $\mu < 1$ for $\alpha$ close to $1$, as required.
\end{proof}

%% file: frogsZd_outlook.tex
We believe that there is a monotone curve separating the transient from the recurrent regime in the phase diagram shown in Figure~\ref{phase_diagram}.

\begin{con}\label{con_critical_curve}
For every dimension $d$ there exists a decreasing function $f_d \colon [0,1] \to [0,1]$ such that the frog model $\fm(d,\pi_{w,\alpha})$ is recurrent for all $w,\alpha \in [0,1]$ such that $w<f_d(\alpha)$ and transient for all $w,\alpha \in [0,1]$ such that $w>f_d(\alpha)$.
\end{con}

Intuitively, the frog model approximates a binary branching random walk for $d \to \infty$ from below, as each frog activates a new frog in every step if there are 'infinitely' many directions to choose from. This leads to the following conjecture.

\begin{con}\label{con_high_d}
 The sequence of functions $(f_d)_{d \in \N}$ is increasing in $d$.
\end{con}

In the proof of Theorem~\ref{thm_d>2_arbitrary_drift_i} we use Lemma~\ref{lemma_recurrence_high_d_percolation_parameter_bound} to show that in the frog model with death a frog cluster is dense with positive probability if the survival probability is larger than $\frac34$ and $d$ is large. 
Indeed, we believe that every infinite frog cluster is dense. Hence, $\fm(d,\pi_{w,\alpha})$ would be recurrent for all $\alpha<1$ if $\fm^*(d-1,\pi_{\text{sym}},1-w)$ has a positive survival probability. Further, we believe that the critical survival probability is decreasing in $d$. See also the discussion in \cite[Chapter~1.2]{AMP02pt}. This would imply that $f_d(1^{-})$ is increasing in $d$.

The comparison with a binary branching random walk raises another question.  Let 
\begin{equation*}
 g \colon [0,1] \to [0,1],\ g(\alpha) = \min\bigl\{1, (2(1-\sqrt{1-\alpha^2}))^{-1}\bigr\}.
\end{equation*}
A binary branching random walk on $\Z^d$ with transition probabilities as in \eqref{transition_function} is recurrent iff $w < g(\alpha)$, see \cite[Section~4]{GM06}. 

\begin{question}
Does the sequence of functions $(f_d)_{d \in \N}$ converge pointwise to $g$ as $d \to \infty$?
\end{question}